\documentclass[12pt]{amsart}
\usepackage{amssymb,amsmath,amsthm,amsfonts,color,longtable}
\usepackage{soul}
\newcommand{\bea}{\begin{eqnarray}}
	\newcommand{\eea}{\end{eqnarray}}
\newcommand{\beaa}{\begin{eqnarray*}}
	\newcommand{\eeaa}{\end{eqnarray*}}
\newcommand{\g}{\mathfrak g}
\newcommand{\h}{\mathfrak h}
\newcommand{\n}{\mathfrak n}
\newcommand{\Z}{\mathbb Z}

\newcommand{\C}{\mathbb C}
\newcommand{\N}{\mathbb N}

\newcommand{\hg}{\hat \g}
\usepackage[hidelinks]{hyperref}

\usepackage[margin=2.5cm]{geometry} 
\newtheorem{theorem}{Theorem}[section]
\newtheorem{main}{Main Theorem}
\newtheorem{lemma}[theorem]{Lemma}

\newtheorem{proposition}[theorem]{Proposition}

\newtheorem{corollary}[theorem]{Corollary}
\newtheorem{remark}[theorem]{Remark}
\newtheorem{rem}[theorem]{Remark}
\newtheorem{teo}[theorem]{Theorem}
\numberwithin{theorem}{section}
\numberwithin{equation}{section}
\allowdisplaybreaks

\begin{document}
	
	
	\title[]{   A method for describing the maximal ideal in universal affine vertex algebras at non-admissible levels}
	
	\author[]{Dra\v zen  Adamovi\' c}
	\address{Department of Mathematics, Faculty of Science \\
		University of Zagreb \\
		Bijeni\v cka 30 }
	\email{adamovic@math.hr}
	
	\author[]{Ozren Per\v se}
	\address{Department of Mathematics, Faculty of Science \\
		University of Zagreb \\
		Bijeni\v cka 30}
	\email{perse@math.hr}
	
	\author[]{Ivana Vukorepa}
	\address{Faculty of Science \\
		University of Split \\
		Ru\dj era Bo\v{s}kovi\'{c}a 33}
	\email{ivukorepa@pmfst.hr} 
	
	\dedicatory{\vspace*{3mm}Dedicated to Professor Eswara Rao on the occasion of his 70th birthday\vspace*{5mm}}
	
	\begin{abstract}
	The problem of determining maximal ideals in universal affine vertex algebras is difficult for levels beyond admissible, since there are no simple character formulas which can be applied.  Here we  investigate when certain quotient  $\mathcal V$ of universal affine vertex algebra $V^k(\g)$ is simple. We present a  new method for proving simplicity of quotients of universal affine vertex algebras  in the case of affine  vertex algebra $L_{k_n}(\mathfrak{sl}_{2n})$ at level $k_n:=-\frac{2n+1}{2}$.   In that way we  describe the maximal ideal in $V^{k_n}(\mathfrak{sl}_{2n})$.
		For that purpose, we use the representation theory of minimal affine $W$-algebra $W^{min}_{k_{n+1}}(\mathfrak{sl}_{2n+2})$ developed in \cite{ACPV-22}.
		
		In particular,  we use the embedding $L_{k_n}(\mathfrak{sl}_{2n}) \subset W^{min}_{k_{n+1}}(\mathfrak{sl}_{2n+2})$ and fusion rules for $L_{k_n}(\mathfrak{sl}_{2n})$--modules.
		We  apply  this result in the cases $n=3,4$  and prove that a maximal ideal is generated by one singular vector of conformal weight $4$.  As a byproduct, we classify irreducible modules in the category $\mathcal{O}$ for the simple affine vertex algebra $L_{-7/2}(\mathfrak{sl}_{6})$.
	\end{abstract}
	\maketitle
	
	\section{Introduction}
	
		Denote by $V^k (\g)$  the universal affine
	vertex algebra of level $k \in \C$ associated to the simple Lie algebra $\g$.  Let $L_k(\g)$ be its simple quotient.
	The basic problem in the representation theory of affine vertex algebras is to classify irreducible $L_k(\g)$--modules. Specially natural problem is the classification of irreducible, highest weight  $L_k(\g)$--modules; which is referred as the classification in the category $\mathcal{O}$.
	It is well-known that for $k$ admissible, which is not positive integer,  $L_k(\g)$ is rational in the category $\mathcal{O}$, meaning that each module in the category $\mathcal O$ is semisimple.
For levels beyond admissible, the classification of irreducible, highest weight  modules is known only in some very special cases (cf. \cite{AP-08, APV, AM-sheets}).
	
	This paper continues our previous research  \cite{ACPV-22, APV} on a family of affine vertex algebras of type $A$ at non-admissible, half-integer levels. In \cite{APV}, we completely determine the  irreducible representations in the category $\mathcal{O}$ for a vertex algebra $L_{-5/2} (\mathfrak{sl}_4)$. We used  singular vector of conformal weight $4$ in the universal affine vertex algebra $V^{-5/2} (\mathfrak{sl}_4)$, and proved that it generates the maximal ideal. In \cite{ACPV-22} we classify irreducible modules for $L_{k_n} (\mathfrak{sl}_{2n})$ in the Kazhdan-Lusztig category $KL_k(\mathfrak{sl}_{2n}) $ for  $k_n:=-\frac{2n+1}{2}$, for $n \geq 3$.	Our approaches uses  precise description  of minimal affine $W$--algebra $W_{k_n}( \mathfrak{sl}_{2n}, \theta)$, which we also obtained. But the methods of \cite{ACPV-22} were not sufficient for determination of maximal ideal in $V^{k_n} (\mathfrak{sl}_{2n})$ and classification of irreducible $L_{k_n} (\mathfrak{sl}_{2n})$--modules in the category $\mathcal{O}$. In the present paper we are focused on these problems.
	
	We also believe that the maximal ideal can be also described by using the fact that $k_n$ is a collapsing level for hook affine $W$--algebras (cf. \cite{AMP}). But the properties of QHR in the  hook cases are still not completely described.

	\subsection*{Main results of the paper} 
	Our main result is the criterion for the simplicity of a certain quotient of $V^{k_n} (\mathfrak{sl}_{2n})$, $k_n= - \frac{2n+1}{2}$.
	\begin{main}\label{main-intro} (cf. Theorem  \ref{main1})
		Let $k_n= - \frac{2n+1}{2}$. Assume that $\mathcal{V}$ is a quotient of $V^{k_{n+1}}(\mathfrak{sl}_{2n+2})$ such that:
		\begin{itemize}
			\item[(1)] The complete list of $\mathcal{V}$--modules in the category $KL_{k_{n+1}}$ is given by the set
			$$ \{  L_{k_{n+1}}(t \omega_1) ,  L_{k_{n+1}}(t \omega_{2n+1}) \ | \ t \in \Z_{\geq 0} \}.$$
			\item[(2)] The maximal ideal of $V^{k_n}(\mathfrak{sl}_{2n})$ is generated by a singular vector $\Omega_n$ of conformal weight~4.
		\end{itemize}
		Then $\mathcal{V} \cong L_{k_{n+1}}(\mathfrak{sl}_{2n+2})$.
	\end{main}
	
	 So our criterion can be  basically interpreted as follows:
	\begin{itemize}
	\item If certain quotient $\mathcal{V}$  of $V^{k_n} (\mathfrak{sl}_{2n})$ has the same irreducible modules in the Kazhdan-Lusztig category as  $L_{k_n} (\mathfrak{sl}_{2n})$ (which were previously classified in \cite{ACPV-22}), then  $\mathcal V \cong~L_{k_n} (\mathfrak{sl}_{2n})$.
	\end{itemize}

	 Let us give few comments on the proof of the main theorem.  First, we prove there is a non-zero homomorphism $L_{k_n}(\g^\natural) \to H_0(\mathcal{V})$, where $H_0(\cdot)$ denotes the minimal QHR functor. From \cite[Theorem 4.1]{ACPV-22} we have following:
	\begin{enumerate}
		\item Minimal $W$--algebra $W_{k_{n+1}}^{min}(\mathfrak{sl}_{2n+2})$  is a simple current extension of $L_{k_n}(\mathfrak{sl}_{2n}) \otimes \mathcal{H} \otimes \mathcal{M}$ where $\mathcal{H}$ is the rank
		one Heisenberg vertex algebra, and $\mathcal{M} = \mathcal{M}(2) $ the singlet vertex algebra for $c=-2$.
		\item Denote $U_i^{(n)} := L_{k_n}(i \omega_1)$, $U_{-i}^{(n)} := L_{k_n}(i \omega_{2n-1})$, for $i \in \Z_{\geq 0}$. The set $\{ U_i^{(n)} \ | \ i \in \Z \}$ provides a complete list of inequivalent irreducible $L_{k_n}(\mathfrak{sl}_{2n})$--modules in $KL_{k_n}$
		with fusion rules $U_i^{(n)} \times U_j^{(n)} = U_{i+j}^{(n)}$.
	\end{enumerate}
	
	So, since there is the embedding  $L_{k_n}(\mathfrak{sl}_{2n}) \subset W^{min}_{k_{n+1}}(\mathfrak{sl}_{2n+2})$, we have $\Omega_n = 0 $ in $\mathcal{W}_{k_{n+1}}$. If $\Omega_n$ is non-zero in $H_0(\mathcal{V})$, then there is a singular vector $v$ in $H_0(\mathcal{V})$ of conformal weight less or equal four. From classification of $\mathcal{V}$--modules in $KL_{k_{n+1}}$, it follows that $v$ is a QHR of a singular vector in $\mathcal{V}$ of $\mathfrak{sl}_{2n+2}$--weight $t \omega_1$ or $t \omega_{2n+1}$. There are a few possibilities, but all of them lead to a contradiction (cf. Lemma \ref{lemma-homo}).

	Next, we consider the commutant  $\mathcal{C}:=( L_{k_n}(\mathfrak{sl}_{2n}) \otimes \mathcal{F}^\ell, H_0(\mathcal{V}))$, where $\mathcal{F}^\ell$ is the Heisenberg vertex algebra as in \cite{ACPV-22}, and decompose $H_0(\mathcal{V})$ as a $L_{k_n}(\mathfrak{sl}_{2n}) \otimes \mathcal{F}^\ell \otimes \mathcal{C}$--module. Then we use fusion rules for $L_{k_n}(\mathfrak{sl}_{2n})$ and Heisenberg vertex algebra to finish the proof, see Proposition \ref{decomposition}.
	
	In Section \ref{sec-classif} we present an application of Main Theorem \ref{main-intro} to the case $V^{-7/2}(\mathfrak{sl}_6)$. First, we determine an explicit formula for the singular vector $v$ in $V^{-7/2}(\mathfrak{sl}_6)$ of conformal weight 4 (cf. Theorem \ref{sing-v-sl6}). This, together with Zhu's theory, enables us to classify irreducible modules in the category $\mathcal{O}$ for the associated quotient $\mathcal{V}=V^{-7/2}(\mathfrak{sl}_6) / \langle v \rangle$ (cf. Theorem \ref{prop-Zhu-O}). As a consequence, we obtain that irreducible $\mathcal{V}$--modules in the category $KL_{-7/2}$ are $\{ L_{-7/2}(t \omega_1) , L_{-7/2}(t \omega_5) \ | \ t \in \Z_{\geq 0} \}$, so $\mathcal{V}$ satisfies the assumption (1) of Main Theorem \ref{main-intro}. Furthermore, in \cite[Section 5]{APV} an explicit formula for the singular vector of conformal weight four in $V^{-5/2}(\mathfrak{sl}_4)$ is given and it is proved that it generates the maximal ideal. So, we can apply Main Theorem \ref{main-intro} and conclude that $\mathcal{V}$ is simple, i.e. $L_{-7/2}(\mathfrak{sl}_6)=\mathcal{V}$. As a consequence, we get that the set $\{ L_{-7/2}(\mu_i(t)) \ | \ i=1, \ldots , 96, \ t \in \C \}$ from Theorem \ref{prop-Zhu-O} provides a complete list of irreducible $L_{-7/2}(\mathfrak{sl}_6)$--modules in the category $\mathcal{O}$.

	\subsection*{Acknowledgements}
	The authors are partially supported by the Croatian Science Foundation under the project IP-2022-10-9006 and by the project “Implementation of cutting-edge research and its application as part of the Scientific Center of Excellence QuantiXLie“, PK.1.1.02, European Union, European Regional Development Fund.
	
	\section{Preliminaries}\label{sec-prelim}

	In this section we introduce the notation and recall some results on affine vertex algebras (cf. \cite{Bo}, \cite{FZ}, \cite{K-inf}, \cite{K98}). We also discuss minimal affine $\mathcal{W}$--algebras (cf. \cite{KW04}, \cite{AKMPP-18}, \cite{ACPV-22}, \cite{Ar05}).
	\subsection{Affine vertex algebras}
	We introduce the following notation:
	\begin{itemize}
		\item Let $\g $ be the simple Lie algebra with Cartan subalgebra $\h$ and triangular decomposition $\g = \frak n_{-} \oplus \frak h \oplus \frak n_{+}$. Fix the highest root $\theta$ of $\g$. Let $\langle \cdot , \cdot \rangle : \g \times \g \rightarrow \C$ be the Killing form, normalized by the condition $\langle \theta, \theta \rangle =2$.
		\item For $\mu\in \h^*$, denote by $V(\mu)$ the irreducible highest-weight $\g$--module with the highest weight $\mu$.
		\item Let us denote by $\alpha_1, \ldots, \alpha_l$ simple roots, by $h_1, \ldots, h_l$ simple coroots ($h_i = \alpha^\vee_i$, for $i = 1, \ldots, l$), and by $\omega_1 , \ldots , \omega_l$	fundametal weights for $\g$ ($l = {\mbox{dim}} \h$).
		\item Let $\hg = \g \otimes \C[t,t^{-1}] \oplus \C K$ be the affine Lie algebra associated to $\g$.
		\item For $\mu \in \h$ and $k \in \C$, denote by $L_k(\mu)$ the irreducible highest weight $\hg$--module with highest weight $\hat{\mu}:=k \Lambda_0 + \mu \in \hat{\h}^*$.
		\item For $k$ non-critical, i.e., $k \neq -h^\vee$, where $h^\vee$ is the dual Coxeter number of $\g$, let $V^k(\g)$ be the universal affine vertex algebra of level $k$ associated to the simple Lie algebra $\g$. Let $L_k(\g)$ be the simple quotient of  $V^k(\g)$.
		\item For any quotient $V$ of $V^k(\g)$, we define the category $KL_k$ of  $V$--modules as in \cite{AKMPP-20}.
	\end{itemize}

	\subsection{Classification of irreducible modules in the category $\mathcal{O}$}\label{subsec-Zhu}

For a vertex algebra $V$, denote by $A(V)$ Zhu’s algebra associated to $V$ (cf. \cite{Zhu96}). Let  $v \in V^k(\g)$ be a singular vector. We denote by $\widetilde{L}_k(\g)=V^k(\g)/\langle v \rangle$ the associated quotient vertex algebra.  Let us recall the method for classification of irreducible $\widetilde{L}_{k}( \mathfrak{g})$--modules in the category $\mathcal{O}$ developed in \cite{A-94, AM}. Since it has been discused and used in several previous papers \cite{APV, AP-08, P1,  P2}, we recall it very briefly here. It is well known that $A(V^k(\g)) \cong \mathcal{U}(\g)$ and $A(\widetilde{L}_k(\g)) \cong \mathcal{U}(\g) / \langle v' \rangle$, where $v'$ is the image od $v \in V^k(\g)$ in $\mathcal{U}(\g)$ (see \cite{FZ} for details). Let $R$ be a $\mathcal{U}(\g)$--submodule of $\mathcal{U}(\g)$ generated by $v'$ under the adjoint action. We denote by $R_0$ its zero-weight space. For $\mu \in \h^*$, let $V(\mu)$ be an irreducible highest weight $\mathcal{U}(\g)$--module with the highest weight vector $v_\mu$. It is easily to see that for each $r \in R_0$ there exists the unique polynomial $p_r \in \mathcal{S}(\h)$ such that $r v_\mu = p_r(\mu) v_\mu$. Set $\mathcal{P}_0 = \{ p_r \in \mathcal{S}(\h) \ | \ r \in R_0 \}$. We have:

	\begin{proposition}\cite{A-94, AM} \label{koro-polinomi} There is a one-to-one correspondence between 
		\begin{enumerate} 
		\item irreducible $\widetilde{L}_{k}( \mathfrak{g})$--modules
			in the category $\mathcal{O}$ (for  $\hat{\mathfrak{g}}$); 

			\item irreducible $A(\widetilde{L}_{k}( \mathfrak{g}))$--modules
			in the category $\mathcal{O}$ (for $\g$)
			\item weights $\mu \in {\frak h}^{*}$ such that
			$p(\mu)=0$ for all $p \in {\mathcal P}_{0}$.
		\end{enumerate}
	\end{proposition}

	\subsection*{Representation theory of $L_k(\g)$} 
	
	Denote by $L_k (\g)$ the simple affine
	vertex algebra of level $k \in \C$ associated to the simple Lie algebra $\g$. 
	It is well-known that for $k$ admissible, which is not positive integer,  $L_k(\g)$ is rational in the category $\mathcal{O}$, meaning that each module in the category $\mathcal O$ is semisimple.
	
	If $k$ is not admissible, but $k + h^\vee$ is still positive rational number, one expects interesting representation theory of $L_k(\g)$.
	Recall that in the case $\g = \mathfrak{sl}_n$, level $k$ is admissible if and only if $ k + n = \frac{p}{q}$, with $p,q \in \N$, $(p,q) = 1$ and $p \geq n$. Then one can restricts to so-called almost admissible levels, i.e. levels of the form $k = -n + \frac{n-1}{q}$, with $q \in \N$, $(n-1, q) = 1$ \cite{ACPV-22}. In the case $k=-1$ it was proved that category $\mathcal{O}$ is not semisimple and there are infinitely many irreducibles, unlike the admissible case \cite{AP-08}. In the case $L_k(\mathfrak{sl}_{2n})$, for $k=-\frac{2n+1}{2}$, modules in the category $KL_k$ are classified and it is proved that $KL_k$ is semisimple, rigid braided tensor category \cite{ACPV-22}, but the category $\mathcal{O}$ has been studied only for $n=2$. Authors in \cite{APV} classified irreducible modules in $\mathcal{O}$ for $L_{-5/2}(\mathfrak{sl_4})$. They are parametrized as an union of 16 lines in $\C^3$ and indecomposable modules are constructed.

	\subsection{Minimal $\mathcal{W}$--algebras}
	\label{subsect-minimal}
	
	Let $\g = \mathfrak{sl}_{m+2}$. We assume $k \neq - h^\vee$. We choose root vectors $e_\theta, e_{-\theta} \in \g$ such that $[e_\theta , e_{-\theta}] = x \in \h, \ [x, e_{\pm \theta}] = \pm e_{\pm \theta}$. The eigenspace decomposition of $ad\ x$ defines a minimal $\frac{1}{2} \Z$--gradation $\g = \g_{-1} \oplus \g_{-1/2} \oplus \g_0 \oplus \g_{1/2} \oplus \g_1$. Let $\g^\natural=\{ a \in \g_0 \ | \ \langle a,x \rangle =0  \}$. We have $\g^\natural = \mathfrak{sl}_m  \oplus \mathcal \C$.
	Let $W^{k}_{min}(\g)$ be the universal minimal $\mathcal{W}$--algebra associated to the triple $(\g, e_{-\theta},k)$, by applying the quantum Hamiltonian reduction functor $H_0$ to the $V^k(\g)$ \cite{KW04}. Let $W_k^{min}(\g)$ be the simple quotient of $W^{k}_{min}(\g)$.
	
	For the dominant integral weight
	$\lambda = \lambda_1 \omega_1 + \cdots+ \lambda_{m+1} \omega_{m+1}$, $\lambda_i \in {\Z}_{\ge 0}$, for $\g$, we define the dominant integral weight for $\mathfrak{sl_m}$ with
	$\bar \lambda = \lambda_2 \bar \omega_1 + \cdots + \lambda_{m} \bar \omega_{m-1}$.
	
	We recall (cf.  \cite{Ar05}) that  
	\begin{itemize}
	
	\item	$H_0$ is an exact functor from the category $KL_k$ to the category of
		$W^{k}_{min}(\g)$--modules.
	\end{itemize}
	 	
	\begin{proposition}\label{prop:min}\cite{KW04,Ar05, ACPV-22}  Assume that  $k \notin {\Z}_{\ge 0}$,  $\lambda$ is dominant integral weight for  $\mathfrak{sl}_{m+2}$ such that $L_k(\lambda)$  a  highest weight module in $KL_k$. Then $H_{0} ( L_k(\lambda))$ is an irreducible, highest weight module for $W^{k}_{min}(\mathfrak{sl}_{m+2})$ with a highest weight vector $\overline {v}_{\lambda}$ such that the highest weight with respect to $\g^\natural = \mathfrak{sl}_m  \oplus \mathcal \C$ is $(  \bar \lambda , \langle \lambda, \omega_1 - \omega_{m+1} \rangle )$. In particular we have:
		\begin{itemize}
			\item $J(0) \overline {v}_{\lambda} =     \langle \lambda, \omega_1 - \omega_{m+1} \rangle   \overline {v}_{\lambda}$. 
			\item $L(0)   \overline {v}_{\lambda} = \left( \frac{ \langle \lambda, \lambda + 2 \rho\rangle}{ 2( k + h^{\vee})} -  \tfrac{1}{2}  \langle  \lambda, \theta \rangle \right)    \overline {v}_{\lambda}, $ 
		\end{itemize}
		where $J$ is the Heisenberg field obtained from 
		$ J \equiv  \frac{1}{m+2} t(-1){\bf 1} $, $ t =  \mbox{diag} ( m,  \underbrace{-2, \dots, -2}_{m}, m ) \in~\g. $
	\end{proposition}

	\subsection{The category $KL_{k}$ via minimal $W$--algebras} 
	Let $n \in \Z_{\ge 2}$, $k_n=-\frac{2n+1}{2}$. We introduce the following notation for the irreducible $L_{k_n}(\mathfrak{sl}_{2n})$--modules:
	$$ U_i^{(n)} = L_{k_n}(i \omega_1), \ U_{-i}^{(n)} = L_{k_n}(i \omega_{2n-1}), \ i \in \Z_{\ge 0}.$$
	We recall the main theorem from \cite{ACPV-22}:
	
	\begin{theorem}\cite[Theorem 4.1]{ACPV-22}For each $n \in \mathbb Z_{\geq 2}$ we have:
		\begin{itemize}
			\item[(1)]   the set  $\{ U_{i} ^{(n)}  \vert \ i \in {\Z} \}$ provides a complete list of inequivalent  irreducible modules in $KL_{k_n} (\mathfrak{sl}_{2n})$ with fusion rules
			$$ U_i ^{(n)} \times U_j ^{(n)} = U_{i+j}^{(n)}, \quad i, j \in {\Z}.$$
			\item[(2)]  $KL_{k_n} (\mathfrak{sl}_{2n})$ is a semisimple rigid braided tensor category.
			\item[(3)] There exists conformal embedding $L_{k_n} (\mathfrak{sl}_{2n}) \otimes \mathcal H \otimes \mathcal M(2) \hookrightarrow  W_{k_{n+1}}^{min}(\mathfrak{sl}_{2n+2})$.
			\item[(4)] $W_{k_{n+1}}^{min}(\mathfrak{sl}_{2n+2})$ is a simple current extension of $ L_{k_n} (\mathfrak{sl}_{2n}) \otimes \mathcal H \otimes \mathcal M(2)$ and we have the following decomposition:
			$$W_{k_{n+1}}^{min}(\mathfrak{sl}_{2n+2}) = \bigoplus_{i \in {\Z}}  U_{i} ^{(n)} \otimes \mathcal F^\ell_{i} \otimes \mathcal M_i,$$
			{where $\mathcal F^\ell_{i}$ is the Fock module of weight $i$ for the rescaled Heisenberg field $h \mapsto \sqrt{\ell} h$, $\ell = - \frac{m}{m+2}$.}
		\end{itemize}
	\end{theorem}

	\section{Maximal ideal in $V^{k}(\mathfrak{sl}_{2n})$ at level $k=-\frac{2n+1}{2}$}\label{section-main}
	As in  \cite{ACPV-22}, for $n \in \Z_{\ge 2}$, $m=2n$, we use the following notation
\begin{itemize}
	\item $k_{n}= -\frac{2n+1}{2}$, $\g_{n} = \mathfrak{sl}_{2n}$,
	\item $\mathcal W^{k_n} = W^{k_n}_{min} (\g_n)$, $\mathcal W_{k_n} = W_{k_n} ^{min} (\g_n)$.
\end{itemize} 
	
	\begin{theorem} \label{main1} 
		Assume that $\mathcal V$ is any quotient of $V^{k_{n+1}}(\g_{n+1})$ with the following properties:
		\begin{itemize}
			\item[(1)] The set
			$$\{ L_{k_{n+1}} (t \omega_1), L_{k_{n+1}} (t \omega_{m+1}) \ \vert \ t \in {\Z}_{\ge 0} \} $$
			provides a complete list of irreducible $\mathcal V$--modules in  the category $KL_{k_{n+1}}$.
			\item[(2)]  The maximal ideal  of  $V^{k_n} (\g_n)$ is generated by a singular vector $\Omega_n$ of conformal weight~$4$.
		\end{itemize}
		Then $\mathcal V \cong L_{k_{n+1}}(\g_{n+1})$.
	\end{theorem}
	

	
	Recall that for dominant integral $\mathfrak{sl}_{m+2}$--weight $\lambda= t \omega_1$ or $\lambda = t \omega_{m+1}$ we have $\bar \lambda = 0$. From Proposition \ref{prop:min} it follows that  $H_0( L_{k_{n+1} } (t \omega_1) )$   and  $H_0( L_{k_{n+1}} (t \omega_{m+1}) )$ are irreducible $\mathcal W^{k_{n+1}} $--modules  of  $\g^{\natural} =\mathfrak{sl}_{m} \oplus{\C}$ highest weight $(0, \frac{m}{m+2} t)$ and conformal weight
$h^{(t)} = \tfrac{t^2}{m+2}   + \tfrac{t}{2}  $.
	
	Assume that $(1)$ and $(2)$ from Theorem \ref{main1} hold.
	
	\begin{lemma}\label{lemma-homo}
		There is a non-zero homomorphism
		$ L_{k_{n}} (\g ^{\natural} ) \rightarrow H_0( \mathcal V)$. 
	\end{lemma}
	\begin{proof}	
		Consider $V^{k_n}(\g_n)$--singular vector $\Omega_n$ as element of $V^{k_n} (\g_n) \subset  \mathcal W^{k_{n+1}}$.  By \cite{ACPV-22}  we know that $\Omega_n \equiv 0$ in $\mathcal W_{k_{n+1}}$. 
		Assume that $\Omega_n $ is non-zero in $H_0(\mathcal V)$.
		By the property (1) of Theorem \ref{main1}, we have that the maximal ideal  $\mathcal I$ of $\mathcal V$, if it is non-zero,  is generated  by certain singular vectors of $\g_{n+1}$--highest weights $t \omega_1$ or $t \omega_{m+1}$ for certain $t \in {\Z}_{\ge 0}$.
		Since QHR functor $H_0$ is exact, we have that $H_0(\mathcal V / \mathcal I ) =\mathcal W_{k_{n+1}}$.
		Therefore, $\Omega_n  \in H_0(\mathcal I)$. We conclude that there is a singular vector in $H_0(\mathcal V)$ of conformal weight which is  less or equal to four, which is a QHR of a singular vector in $\mathcal V$  of the weight   $t \omega_1$ or $t \omega_{m+1}$ for certain $t \in {\Z}_{\ge 0}$.
		
		Consider the generator $J$ of $\mathcal W^{k_{n+1}}$. Then $J(0)$ acts semisimply on  $H_0(\mathcal V)$ with integral eigenvalues.
		We need to find all possible $t \in {\Z}_{\ge 0}$ such that the QHR of a singular vector with weight  $t \omega_1$ or $t \omega_{m+1}$ has  integer conformal weight less or equal to four in $H_0(\mathcal V)$. Moreover, using action of $J(0)$ we conclude that 
		$\frac{m}{m+2} t = \frac{n}{n+1} t \in {\Z}$.  We conclude that $t = q (n+1)$ for certain $q \in {\Z}_{\ge 0}$ and $J(0)$ acts  on this singular vector as  $q n \mbox{Id}$.  Then
		$$h^{( t)} = (q^2 + q) \frac{n+1}{2} \ge (n+1)$$
		Now the condition that $\vert  h^{( t)} \vert \le 4$, gives that the only possibilities  are
		\begin{itemize}
			\item[(i)]  $q=1$, $n =2$, and the corresponding conformal weight is $h^{(3)} = 3$.
			\item[(ii)] $q =1$, $n=3$, and   the corresponding conformal weight is $h^{(4)} = 4$.
		\end{itemize}
		The case (ii) is not possible, since the $J(0)$--weight of the singular vector  is $3$ and each product of three $G$--generators in  the basis must have conformal weight $>4$. We conclude that 
		the only solution is $n=2$ and  a singular vector of conformal weight $3$ and $J(0)$--weight $2$,  so
		such vector must belong to  the vector space $ W= \mbox{span}_{\C} \{ :G^{+}_{i} G^{+} _j: \ \vert \ i, j =1, \dots, m\}$.
		  From Proposition \ref{prop:min} follows that for $\lambda= t \omega_1$ or $\lambda= t \omega_{m+1}$ we have $\bar \lambda = 0$, so its $\mathfrak{sl}_{m}$--highest weight must be $0$. 
		
		But any $\mathfrak{sl}_{m}$--singular vectors in $ W$ should give a component in the tensor product of $V(\omega_1) \otimes V(\omega_1)$. But since 
		$V(\omega_1) \otimes V(\omega_1) = V(2 \omega_1)   \oplus V(\omega_2)$, we see that $W$ doesn't have $1$--dimensional submodule.  
		We get a contradiction. Therefore $\Omega_n \equiv 0$ in $H_0(\mathcal V)$. The proof follows from the fact that $\Omega_n$ generates the maximal ideal in $V^{k_n} (\g_n)$.
		
	\end{proof}
	
	Consider the commutant  $\mathcal C: = \mbox{Com} (L_{k_n} (\g_n) \otimes \mathcal F^{\ell},H _0( \mathcal V) )$, where $\mathcal F^{\ell}$ is the Heisenberg vertex algebra as in \cite{ACPV-22}. 
	 The following proposition holds.
	
	\begin{proposition} \label{decomposition} We have the following decomposition of  $H_0(\mathcal V)$ as a  $L_{k_n} (\g_n) \otimes \mathcal F^{\ell} \otimes \mathcal C$--module:
		$$H_0( \mathcal V )= \sum_{i \in \Z} U_i ^{(n)}  \otimes  \mathcal F^{\ell} _{i} \otimes  \mathcal C_{i}, $$
		where $U_i ^{(n)}$ is an  irreducible   $L_{k_n} (\g_n)$--module as in \cite{ACPV-22},  $\mathcal F^{\ell} _{i}$ is a module for the Heisenberg vertex algebra   $\mathcal F^{\ell} $  generated by $J$ such such that $J(0)$ acts as $i \mbox{Id}$, and $C_{i}$ is a certain $\mathcal C$--module.
	\end{proposition}
	\begin{proof}
		Note that $H_0( \mathcal V )$ is generated by $L_{k_{n}} (\g_{n}) \otimes \mathcal F ^{\ell}  \otimes \mathcal C +  U_1 ^{(n)} \otimes  \mathcal F ^{\ell} _1  \otimes \mathcal C_1 + U_{-1} ^{(n)} \otimes  \mathcal F ^{\ell} _{-1}  \otimes \mathcal C_{-1}$. The claim now follows from fusion rules for $L_{k_{n}} (\g_{n})$--modules from  \cite{ACPV-22}  and fusion rules for the Heisenberg vertex algebra.
	\end{proof}
	
	\begin{proof}[Proof of Theorem \ref{main1}] If $\mathcal V$ is not simple, then $H_0( \mathcal V )$ must contain a singular vector of $\g_{n+1}^{\natural} = \g_n \times {\C}$ highest weight $(0, b)$ where $b \ne 0$.  But  this is impossible since such singular vector  does not appear in the decomposition of $H_0( \mathcal V )$  as  a $L_{k_n} (\g_n) \otimes \mathcal F^{\ell} \otimes \mathcal C$--module.
	\end{proof}

	\begin{rem}
		In the following we show that Theorem \ref{main1} can be applied to $V^{-7/2}(\mathfrak{sl}_{6})$ and $V^{-9/2}(\mathfrak{sl}_{8})$ (see Section \ref{sec-classif} and Section \ref{sl8}). In particular, we obtain that the maximal ideal in those vertex algebras is generated by one singular vector of conformal weight 4. For the case $V^{-5/2}(\mathfrak{sl}_{4})$, that was proved in \cite{APV}.	
	\end{rem}

	\section{Affine vertex algebra associated to $\widehat{\mathfrak{sl}_{6}}$ at level $-7/2$}\label{sec-quotient}
	Throughout this section we denote by $\g$ the simple Lie algebra $\mathfrak{sl}_{6}$.
	
	\begin{teo}\label{sing-v-sl6}
		There is a singular vector $v$ in $V^{-7/2}(\g)$ of weight $-\frac{7}{2}\Lambda_0 - 4\delta + \omega_2 + \omega_4$. Its explicit formula is given in Appendix A.
	\end{teo}
	\begin{proof}
		Direct verification of relations $e_{i,i+1}(0)v=0$ for $i=1,\ldots,5$ and $f_{1,6}(1)v=0$ using \textit{Mathematica}.
	\end{proof}
	
	Let us denote by $$ \widetilde{L}_{-7/2}(\g) = V^{-7/2}(\g) / \langle v \rangle$$ the associated quotient vertex algebra.
	
	In the next propositon we determine Zhu’s algebra of $\widetilde{L}_{-7/2}(\g)$.
	
	\begin{proposition}
		Zhu’s algebra $A(\widetilde{L}_{-7/2}(\g))$ is isomorphic to $\mathcal{U}(\g)/ \langle v' \rangle$, where $\langle v' \rangle$ is a two-sided ideal
		in $\mathcal{U}(\g)$ generated by the following vector $v'$:
		\begin{eqnarray*}
			&&v'=\frac{7}{2} e_{1,5}  e_{2,6} - \frac{7}{2} e_{1,6} e_{2, 5} - \frac{4}{3} e_{1,2} e_{2, 5} e_{2,6} + 
			\frac{1}{3} e_{1,3} e_{2, 5} e_{3,6} - 
			\frac{5}{3} e_{1,3} e_{2,6} e_{3,5}+ 
			\frac{5}{3} e_{1,4} e_{2,5} e_{4,6} \\
			&&- 
			3 e_{1,4} e_{2,6} e_{4,5}- 
			\frac{1}{3} e_{1,5} e_{2,3} e_{3,6}- 
			\frac{5}{3} e_{1,5} e_{2,4} e_{4,6} + 
			\frac{5}{3} e_{1,6} e_{2,3} e_{3,5} + 
			3 e_{1,6} e_{2,4} e_{4,5} + \frac{4}{3}  f_{1,2} e_{1,5} e_{1,6} \\
			&&-
			h_1 e_{1,5} e_{2,6} - \frac{1}{3}  h_1 e_{1,6} e_{2, 5} - 
			\frac{4}{15} h_2 e_{1,5} e_{2,6} + \frac{4}{15} h_2 e_{1,6} e_{2, 5} + 
			\frac{8}{15} h_3 e_{1,5} e_{2,6} - \frac{8}{15} h_3 e_{1,6} e_{2, 5} \\
			&&+ 
			\frac{12}{5} h_4 e_{1,5} e_{2,6} - \frac{12}{5} h_4 e_{1,6} e_{2, 5} + 
			h_5 e_{1,5} e_{2,6} - h_5 e_{1,6} e_{2, 5} + 
			\frac{2}{3} e_{1,2} e_{2,3} e_{2, 5} e_{3,6}- 
			\frac{2}{3} e_{1,2} e_{2,3} e_{2,6} e_{3,5}\\
			&&+ 
			\frac{2}{3} e_{1,2} e_{2,4} e_{2, 5} e_{4,6} - 
			\frac{2}{3} e_{1,2} e_{2,4} e_{2,6} e_{4,5} + 
			\frac{2}{3} e_{1,2} e_{2, 5}^2 e_{5,6} + 
			2 e_{1,3} e_{2,4} e_{3,5} e_{4,6} - 
			2 e_{1,3} e_{2,4}  e_{3,6} e_{4,5} \\
			&&- 
			\frac{4}{3} e_{1,3} e_{2, 5} e_{3,4} e_{4,6} + 
			\frac{2}{3} e_{1,3} e_{2, 5} e_{3,5} e_{5,6} + 
			\frac{4}{3} e_{1,3} e_{2,6} e_{3,4} e_{4,5} - 
			2 e_{1,4} e_{2,3} e_{3,5} e_{4,6} + 
			2 e_{1,4} e_{2,3}  e_{3,6} e_{4,5} \\
			&&+ 
			\frac{2}{3} e_{1,4} e_{2,5} e_{4,5} e_{5,6} + 
			\frac{4}{3} e_{1,5} e_{2,3} e_{3,4} e_{4,6} - 
			\frac{2}{3} e_{1,5} e_{2,3} e_{3,5} e_{5,6} - 
			\frac{2}{3} e_{1,5} e_{2,4} e_{4,5} e_{5,6}- 
			\frac{4}{3} e_{1,6} e_{2,3} e_{3,4} e_{4,5} \\
			&&- 
			\frac{2}{5} f_{1,2} e_{1,2} e_{1,5} e_{2,6} + 
			\frac{2}{5} f_{1,2} e_{1,2} e_{1,6} e_{2,5} - 
			\frac{2}{3} f_{1,2} e_{1,3} e_{1,5} e_{3,6} + 
			\frac{2}{3} f_{1,2} e_{1,3} e_{1,6} e_{3,5} - 
			\frac{2}{3} f_{1,2} e_{1,4} e_{1,5} e_{4,6} \\
			&&+ 
			\frac{2}{3} f_{1,2} e_{1,4} e_{1,6} e_{4,5} - 
			\frac{2}{3} f_{1,2} e_{1,5}^2  e_{5,6} + 
			\frac{4}{15} f_{1,3} e_{1,3} e_{1,5} e_{2,6} - 
			\frac{4}{15} f_{1,3} e_{1,3} e_{1,6} e_{2,5} + 
			\frac{4}{15} f_{1,4} e_{1,4} e_{1,5} e_{2,6} \\
			&&- 
			\frac{4}{15} f_{1,4} e_{1,4} e_{1,6} e_{2,5} + 
			\frac{4}{15} f_{1,5} e_{1,5}^2 e_{2,6} - 
			\frac{4}{15} f_{1,5} e_{1,5} e_{1,6} e_{2,5} + 
			\frac{4}{15} f_{1,6} e_{1,5} e_{1,6} e_{2,6} - 
			\frac{4}{15} f_{1,6} e_{1,6}^2 e_{2,5} \\
			&&+ 
			\frac{4}{15} f_{2,3} e_{1,5} e_{2,3} e_{2,6} - 
			\frac{4}{15} f_{2,3} e_{1,6} e_{2,3} e_{2,5} + 
			\frac{4}{15} f_{2,4} e_{1,5} e_{2,4} e_{2,6} - 
			\frac{4}{15} f_{2,4} e_{1,6} e_{2,4} e_{2,5} + 
			\frac{4}{15} f_{2,5} e_{1,5} e_{2,5} e_{2,6} \\
			&&- 
			\frac{4}{15} f_{2,5} e_{1,6} e_{2,5}^2 + 
			\frac{4}{15} f_{2,6} e_{1,5} e_{2,6}^2  - 
			\frac{4}{15} f_{2,6} e_{1,6} e_{2,5} e_{2,6} - 
			\frac{4}{3} f_{3,4} e_{1,4} e_{2,5} e_{3,6} + 
			\frac{4}{3} f_{3,4} e_{1,4} e_{2,6} e_{3,5} \\
			&&+ 
			\frac{4}{3} f_{3,4} e_{1,5} e_{2,4} e_{3,6} - 
			\frac{16}{15} f_{3,4} e_{1,5} e_{2,6} e_{3,4} - 
			\frac{4}{3} f_{3,4} e_{1,6} e_{2,4} e_{3,5} + 
			\frac{16}{15} f_{3,4} e_{1,6} e_{2,5} e_{3,4} + 
			\frac{4}{15} f_{3,5} e_{1,5} e_{2,6} e_{3,5} \\
			&&- 
			\frac{4}{15} f_{3,5} e_{1,6} e_{2,5} e_{3,5} + 
			\frac{4}{15} f_{3,6} e_{1,5} e_{2,6} e_{3,6} - 
			\frac{4}{15} f_{3,6} e_{1,6} e_{2,5} e_{3,6} + 
			\frac{4}{15} f_{4,5} e_{1,5} e_{2,6} e_{4,5} - 
			\frac{4}{15} f_{4,5} e_{1,6} e_{2,5} e_{4,5} \\
			&&+ 
			\frac{4}{15} f_{4,6} e_{1,5} e_{2,6} e_{4,6} - 
			\frac{4}{15} f_{4,6} e_{1,6} e_{2,5} e_{4,6} - 
			\frac{2}{3} f_{5,6} e_{1,2} e_{2,6}^2 - 
			\frac{2}{3} f_{5,6} e_{1,3} e_{2,6} e_{3,6} - 
			\frac{2}{3} f_{5,6} e_{1,4} e_{2,6} e_{4,6} \\
			&&- 
			\frac{2}{5} f_{5,6} e_{1,5} e_{2,6} e_{5,6} + 
			\frac{2}{3} f_{5,6} e_{1,6} e_{2,3} e_{3,6} + 
			\frac{2}{3} f_{5,6} e_{1,6} e_{2,4}  e_{4,6} + 
			\frac{2}{5} f_{5,6} e_{1,6} e_{2,5} e_{5,6} + 
			\frac{2}{3} f_{5,6} f_{1,2} e_{1,6}^2 \\
			&&+ 
			\frac{2}{3} h_1 e_{1,5} e_{2,3} e_{3,6} + 
			\frac{2}{3} h_1 e_{1,5} e_{2,4} e_{4,6} + 
			\frac{2}{3} h_1 e_{1,5} e_{2,5} e_{5,6} - 
			\frac{2}{3} h_1 e_{1,6} e_{2,3} e_{3,5} - 
			\frac{2}{3} h_1 e_{1,6} e_{2,4} e_{4,5} \\
			&&- 
			\frac{2}{3} h_1 f_{5,6} e_{1,6} e_{2,6} + 
			\frac{2}{5} h_1 h_2 e_{1,5} e_{2,6} - 
			\frac{2}{5} h_1 h_2 e_{1,6} e_{2,5} + 
			\frac{2}{15} h_1 h_3 e_{1,5} e_{2,6} - 
			\frac{2}{15} h_1 h_3 e_{1,6} e_{2,5}\\
			&& - 
			\frac{2}{15} h_1 h_4 e_{1,5} e_{2,6} + 
			\frac{2}{15} h_1 h_4 e_{1,6} e_{2,5} - 
			\frac{2}{5} h_1 h_5 e_{1,5} e_{2,6} - 
			\frac{4}{15} h_1 h_5 e_{1,6} e_{2,5} - 
			\frac{2}{3} h_2 e_{1,3} e_{2,5} e_{3,6} \\
			&&+ 
			\frac{2}{3} h_2 e_{1,3}  e_{2,6} e_{3,5} - 
			\frac{2}{3} h_2 e_{1,4} e_{2,5} e_{4,6} + 
			\frac{2}{3} h_2 e_{1,4} e_{2,6} e_{4,5} + 
			\frac{2}{3} h_2 e_{1,5} e_{2,3} e_{3,6} + 
			\frac{2}{3} h_2 e_{1,5} e_{2,4} e_{4,6} \\
			&&- 
			\frac{2}{3} h_2 e_{1,6} e_{2,3} e_{3,5} - 
			\frac{2}{3} h_2 e_{1,6} e_{2,4} e_{4,5} + 
			\frac{2}{5} h_2 h_2 e_{1,5} e_{2,6} - 
			\frac{2}{5} h_2 h_2 e_{1,6} e_{2,5} + 
			\frac{4}{15} h_2 h_3 e_{1,5} e_{2,6} \\
			&&- 
			\frac{4}{15} h_2 h_3 e_{1,6} e_{2,5} - 
			\frac{4}{15} h_2 h_4 e_{1,5} e_{2,6} + 
			\frac{4}{15} h_2 h_4 e_{1,6} e_{2,5} - 
			\frac{2}{15} h_2 h_5 e_{1,5} e_{2,6} + 
			\frac{2}{15} h_2 h_5 e_{1,6} e_{2,5} \\
			&&- 
			\frac{2}{3} h_3 e_{1,3} e_{2,5} e_{3,6} + 
			\frac{2}{3} h_3 e_{1,3} e_{2,6} e_{3,5} + 
			\frac{2}{3} h_3 e_{1,4} e_{2,5} e_{4,6} - 
			\frac{2}{3} h_3 e_{1,4} e_{2,6} e_{4,5} + 
			\frac{2}{3} h_3 e_{1,5} e_{2,3} e_{3,6} \\
			&&- 
			\frac{2}{3} h_3 e_{1,5} e_{2,4} e_{4,6} - 
			\frac{2}{3} h_3 e_{1,6} e_{2,3} e_{3,5} + 
			\frac{2}{3} h_3 e_{1,6} e_{2,4} e_{4,5} - 
			\frac{2}{15} h_3 h_3 e_{1,5} e_{2,6} + 
			\frac{2}{15} h_3 h_3 e_{1,6} e_{2,5} \\
			&&+ 
			\frac{4}{15} h_3 h_4 e_{1,5} e_{2,6} - 
			\frac{4}{15} h_3 h_4 e_{1,6} e_{2,5} + 
			\frac{2}{15} h_3 h_5 e_{1,5}  e_{2,6} - 
			\frac{2}{15} h_3 h_5 e_{1,6} e_{2,5} + 
			\frac{2}{3} h_4 e_{1,3} e_{2,5} e_{3,6} \\
			&&- 
			\frac{2}{3} h_4 e_{1,3} e_{2,6} e_{3,5} + 
			\frac{2}{3} h_4 e_{1,4} e_{2,5} e_{4,6} - 
			\frac{2}{3} h_4 e_{1,4} e_{2,6} e_{4,5}- 
			\frac{2}{3} h_4 e_{1,5} e_{2,3} e_{3,6} - 
			\frac{2}{3} h_4 e_{1,5} e_{2,4} e_{4,6} \\
			&&+ 
			\frac{2}{3} h_4 e_{1,6} e_{2,3} e_{3,5} + 
			\frac{2}{3} h_4 e_{1,6} e_{2,4} e_{4,5} + 
			\frac{2}{5} h_4 h_4 e_{1,5} e_{2,6} - 
			\frac{2}{5} h_4 h_4 e_{1,6} e_{2,5} + 
			\frac{2}{5} h_4 h_5 e_{1,5} e_{2,6} \\
			&&- 
			\frac{2}{5} h_4 h_5 e_{1,6} e_{2,5} - 
			\frac{2}{3} h_5 e_{1,2} e_{2,5} e_{2,6} - 
			\frac{2}{3} h_5 e_{1,3} e_{2,6} e_{3,5} - 
			\frac{2}{3} h_5 e_{1,4} e_{2,6} e_{4,5} + 
			\frac{2}{3} h_5 e_{1,6} e_{2,3} e_{3,5} \\
			&&+ 
			\frac{2}{3} h_5 e_{1,6} e_{2,4} e_{4,5}+ 
			\frac{2}{3} h_5 f_{1,2} e_{1,5} e_{1,6}.
		\end{eqnarray*}
	\end{proposition}
	
	\begin{proof}
		By direct calulation, one obtains that $v'$ is the image of singular vector $v$ in $\mathcal{U}(\g)$. The proof now follows from \cite[Proposition 1.4.2, Theorem 3.1.1]{FZ}.
	\end{proof}
	
	\section{Classification of irreducible $\widetilde{L}_{-7/2}(\mathfrak{sl}_{6})$--modules in category $\mathcal{O}$}\label{sec-classif}
	
	Let $\g=\mathfrak{sl}_{6}$. In this section we study representation theory of the vertex algebra $\widetilde{L}_{-7/2}(\g)$ defined in Section \ref{sec-quotient}. As a consequence, we obtain that $\widetilde{L}_{-7/2}(\g)$ satisfies assumptions of Theorem \ref{main1}.
	
	In the following lemma we determine the basis of the space $\mathcal{P}_0$  defined in Subsection \ref{subsec-Zhu}. 
	
	\begin{lemma}\label{lemma-polinomi}
		We have $$\mathcal{P}_0 = span_{\C} \{ p_i \ | \ i= 1, \ldots , 9\},$$
		where
		\begin{eqnarray*}
			&&p_1(h) =h_1 h_3 ( 65 + 10 h_1 + 12 h_2 - 8 h_1 h_2 - 8 h_2^2 + 66 h_3 + 4 h_1 h_3 + 8 h_2 h_3 + 
			16 h_3^2 + 112 h_4 \\
			&&\qquad+ 16 h_1 h_4 + 32 h_2 h_4 + 48 h_3 h_4 + 32 h_4^2 + 
			40 h_5 + 8 h_1 h_5 + 16 h_2 h_5 + 24 h_3 h_5 + 32 h_4 h_5 ) , \\
			&&p_2(h)=h_1 h_4 (45 + 10 h_1 + 28 h_2 + 8 h_1 h_2 + 8 h_2^2 + 84 h_3 + 16 h_1 h_3 + 
			32 h_2 h_3 + 24 h_3^2 + 38 h_4 \\
			&&\qquad+ 4 h_1 h_4 + 8 h_2 h_4 + 32 h_3 h_4 + 
			8 h_4^2 - 8 h_1 h_5 - 16 h_2 h_5 + 16 h_3 h_5 + 8 h_4 h_5),\\
			&&p_3(h)= h_1 h_5 (75 + 30 h_1 + 68 h_2 + 8 h_1 h_2 + 8 h_2^2 + 84 h_3 + 16 h_1 h_3 + 
			32 h_2 h_3 + 24 h_3^2 + 68 h_4 \\
			&&\qquad+ 24 h_1 h_4 + 48 h_2 h_4 + 32 h_3 h_4 + 
			8 h_4^2 + 30 h_5 + 12 h_1 h_5 + 24 h_2 h_5 + 16 h_3 h_5 + 8 h_4 h_5),\\
			&&p_4(h)=h_2 h_4 (75 + 40 h_1 + 58 h_2 + 8 h_1 h_2 + 8 h_2^2 + 84 h_3 + 16 h_1 h_3 + 
			32 h_2 h_3 + 24 h_3^2 + 58 h_4 \\
			&&\qquad+ 24 h_1 h_4 + 28 h_2 h_4 + 32 h_3 h_4 + 
			8 h_4^2 + 40 h_5 + 32 h_1 h_5 + 24 h_2 h_5 + 16 h_3 h_5 + 8 h_4 h_5),\\
			&&p_5(h)=h_2 h_5 (45 + 38 h_2 + 8 h_1 h_2 + 8 h_2^2 + 84 h_3 + 16 h_1 h_3 + 32 h_2 h_3 + 24 h_3^2 + 28 h_4 \\
			&&\qquad
			- 16 h_1 h_4 + 8 h_2 h_4 + 32 h_3 h_4 + 8 h_4^2 + 10 h_5 - 
			8 h_1 h_5 + 4 h_2 h_5 + 16 h_3 h_5 + 8 h_4 h_5),\\
			&&p_6(h)=h_3 h_5 (65 + 40 h_1 + 112 h_2 + 32 h_1 h_2 + 32 h_2^2 + 66 h_3 + 24 h_1 h_3 + 
			48 h_2 h_3 + 16 h_3^2 + 12 h_4 \\
			&&\qquad
			+ 16 h_1 h_4 + 32 h_2 h_4 + 8 h_3 h_4 - 
			8 h_4^2 + 10 h_5 + 8 h_1 h_5 + 16 h_2 h_5 + 4 h_3 h_5 - 8 h_4 h_5),\\
			&&p_7(h)=h_2 (\frac{7}{2} + h_1 + h_2 +  h_3 +  h_4 +  h_5) (15 + 15 h_1 + 21 h_2 + 
			6 h_1 h_2 + 6 h_2^2 + 23 h_3 + 12 h_1 h_3 \\
			&&\qquad
			+ 14 h_2 h_3 + 8 h_3^2 + h_4 + 
			8 h_1 h_4 + 6 h_2 h_4 + 4 h_3 h_4 - 4 h_4^2 - 5 h_5 + 4 h_1 h_5 - 2 h_2 h_5 - 8 h_3 h_5 - 4 h_4 h_5),\\
			&&p_8(h)= h_3(\frac{7}{2} + h_1 + h_2 +  h_3 +  h_4 +  h_5)(5 - 5 h_1 + 9 h_2 + 4 h_1 h_2 + 4 h_2^2 + 7 h_3 - 2 h_1 h_3 + 6 h_2 h_3 \\
			&&\qquad
			+ 2 h_3^2 + 9 h_4 - 8 h_1 h_4 + 4 h_2 h_4 + 6 h_3 h_4 + 4 h_4^2 - 5 h_5 - 
			4 h_1 h_5 - 8 h_2 h_5 - 2 h_3 h_5 + 4 h_4 h_5),\\
			&&p_9(h)=h_4 (\frac{7}{2} + h_1 + h_2 +  h_3 +  h_4 +  h_5)(15 - 5 h_1 + h_2 - 4 h_1 h_2 - 4 h_2^2 + 23 h_3 - 8 h_1 h_3 + 4 h_2 h_3 \\
			&&\qquad+ 
			8 h_3^2 + 21 h_4 - 2 h_1 h_4 + 6 h_2 h_4 + 14 h_3 h_4 + 6 h_4^2 + 15 h_5 + 4 h_1 h_5 + 8 h_2 h_5 + 12 h_3 h_5 + 6 h_4 h_5).
		\end{eqnarray*}
	\end{lemma}
	
	\begin{proof}
		We consider $\mathcal{U}(\g)$–submodule $R$ of $\mathcal{U}(\g)$ generated by vector $v'$ under the adjoint action $X_L f = [X, f]$, for $X \in \g, \ f \in \mathcal{U}(\g)$. It is clearly isomorphic to $V(\omega_2 + \omega_4)$. From \cite[Lemma 5.2.]{AP-08} it follows dim$R_0=9$, which implies dim$\mathcal{P}_0 \leq 9$.
		Using the explicit formula for $v'$, one obtains:
		\begin{eqnarray*}
			(f_{1,2} f_{2,3} f_{2,4} f_{3,6} f_{4,5})_L v' &&\hspace*{-5mm}\in \frac{1}{30} p_1(h) + \n_-\mathcal{U}(\g) + \mathcal{U}(\g) \n_+,\\
			(f_{1,2} f_{2,4} f_{2,5} f_{4,6})_L v' &&\hspace*{-5mm}\in \frac{1}{30} p_2(h) + \n_-\mathcal{U}(\g) + \mathcal{U}(\g) \n_+,\\
			(f_{1,2} f_{2,5} f_{2,6})_L v' &&\hspace*{-5mm}\in \frac{1}{30} p_3(h) + \n_-\mathcal{U}(\g) + \mathcal{U}(\g) \n_+,\\
			(f_{1,5} f_{2,3} f_{3,6} + f_{1,4} f_{2,3} f_{3,6} f_{4,5})_L v' &&\hspace*{-5mm}\in \frac{1}{30} p_4(h) + \n_-\mathcal{U}(\g) + \mathcal{U}(\g) \n_+,\\
			(f_{1,3}f_{2,6}f_{3,5} - f_{1,2} f_{2,5} f_{2,6} - f_{1,3} f_{2,5} f_{3,4} f_{4,6})_L v'&&\hspace*{-5mm}\in \frac{1}{30} p_5(h) + \n_-\mathcal{U}(\g) + \mathcal{U}(\g) \n_+,\\
			(f_{1,3} f_{2,6} f_{3,5} - f_{1,4} f_{4,5} f_{2,6} + f_{1,3} f_{2,5} f_{3,6} - f_{1,4} f_{2,5} f_{4,6})_L v' &&\hspace*{-5mm}\in \frac{1}{30} p_6(h) + \n_-\mathcal{U}(\g) + \mathcal{U}(\g) \n_+,\\
			(f_{1,6} f_{2,3} f_{3,5})_L v' &&\hspace*{-5mm}\in \frac{1}{15} p_7(h) + \n_-\mathcal{U}(\g) + \mathcal{U}(\g) \n_+,\\
			(f_{1,6} f_{2,3} f_{3,5} - f_{1,6} f_{2,4} f_{4,5} )_L v'&&\hspace*{-5mm}\in \frac{1}{15} p_8(h) + \n_-\mathcal{U}(\g) + \mathcal{U}(\g) \n_+,\\
			(f_{1,6} f_{2,5} + f_{1,6} f_{2,4} f_{4,5})_L v' &&\hspace*{-5mm}\in \frac{1}{15} p_9(h) + \n_-\mathcal{U}(\g) + \mathcal{U}(\g) \n_+.
		\end{eqnarray*}
		Since polynomials $p_1, \ldots , p_9$ are linearly independent, we obtain the claim of Lemma.
	\end{proof}
	
	The following theorem gives the classification of irreducible $\widetilde{L}_{-7/2}(\g)$–modules in category $\mathcal{O}$.
	
	\begin{theorem}\label{prop-Zhu-O}
		The complete list of irreducible $\widetilde{L}_{-7/2}(\g)$–modules in category $\mathcal{O}$ is given by the set: \begin{equation}\label{eq-hw-katO}
			\{ L_{-7/2}(\mu_i (t)) \ | \ i=1, \ldots , 96, \ t \in \C  \},
		\end{equation}
		where:
		\begin{longtable}{ p{7.5cm} p{9cm} }
			$\mu_1(t)=t\omega_1$, & $\mu_{49}(t)=\frac{1}{2}\omega_1 - \frac{3}{2}\omega_2 + t \omega_5$,  \\ 
			$\mu_2(t)=t\omega_5$, & 	$\mu_{50}(t)=-\frac{1}{2}\omega_1 - \frac{3}{2}\omega_3 + t \omega_5$,  \\  
			$\mu_3(t)=t\omega_1+(-t-\frac{7}{2} ) \omega_{2}$, & 	$\mu_{51}(t)=-\frac{3}{2}\omega_1 - \frac{3}{2}\omega_4 + t \omega_5$,  \\ 
			$\mu_4(t)=t\omega_2+(-t-\frac{7}{2} ) \omega_{3}$, & 	$\mu_{52}(t)=-\frac{1}{2}\omega_2 - \frac{1}{2}\omega_3 + t \omega_5 $, \\ 
			$\mu_5(t)=t\omega_3+(-t-\frac{7}{2} ) \omega_{4}$, & $\mu_{53}(t)=-\frac{3}{2}\omega_2 - \frac{1}{2}\omega_4 + t \omega_5$,  \\ 
			$\mu_6(t)=t\omega_4+(-t-\frac{7}{2} ) \omega_{5}$, & $\mu_{54}(t)=-\frac{3}{2}\omega_3 + \frac{1}{2}\omega_4 + t \omega_5$, \\ 
			$\mu_7(t)=t\omega_1+(-t-1 ) \omega_{2}$, & $\mu_{55}(t)= t \omega_1 - \frac{3}{2}\omega_2 + \frac{1}{2}\omega_3 - \frac{3}{2} \omega_4$,  \\ 
			$\mu_8(t)=t\omega_2+(-t-1 ) \omega_{3}$, & $\mu_{56}(t)=t \omega_1 - \frac{1}{2}\omega_2 - \frac{1}{2}\omega_3 - \frac{3}{2} \omega_5$   \\
			$\mu_9(t)=t\omega_3+(-t-1 ) \omega_{4}$, & $\mu_{57}(t)=t \omega_1 - \frac{3}{2}\omega_2 - \frac{1}{2}\omega_4 - \frac{1}{2} \omega_5$,  \\ 
			$\mu_{10}(t)=t\omega_4+(-t-1 ) \omega_{5}$, & 	$\mu_{58}(t)=t \omega_1 - \frac{3}{2}\omega_3 + \frac{1}{2}\omega_4 - \frac{3}{2} \omega_5 $,  \\  
			$\mu_{11}(t)=t\omega_1-\frac{5}{2} \omega_{2}$, & 	$\mu_{59}(t)=- \frac{3}{2}\omega_1 + \frac{1}{2}\omega_2 - \frac{3}{2} \omega_3 + t \omega_5$,  \\ 
			$\mu_{12}(t)=t\omega_1-\frac{5}{2} \omega_{3}$, & 	$\mu_{60}(t)=- \frac{1}{2}\omega_1 - \frac{1}{2}\omega_2 - \frac{3}{2} \omega_4 + t \omega_5$, \\ 
			$\mu_{13}(t)=t\omega_1-\frac{5}{2} \omega_{4}$, & 	$\mu_{61}(t)=- \frac{3}{2}\omega_1 - \frac{1}{2}\omega_3 - \frac{1}{2} \omega_4 + t \omega_5$,  \\ 
			$\mu_{14}(t)=t\omega_1-\frac{5}{2} \omega_{5}$, & $\mu_{62}(t)=- \frac{3}{2}\omega_2 + \frac{1}{2}\omega_3 - \frac{3}{2} \omega_4 + t \omega_5$, \\ 
			$\mu_{15}(t)= -\frac{5}{2} \omega_{1} + t \omega_5$, & $\mu_{63}(t)= t \omega_1 +(  -\frac{5}{2} - t ) \omega_2 + \frac{1}{2} \omega_3 - \frac{3}{2} \omega_4$,  \\ 
			$\mu_{16}(t)= -\frac{5}{2} \omega_{2} + t \omega_5$, & $\mu_{64}(t)=t \omega_1 +(  -\frac{3}{2} - t ) \omega_2 - \frac{1}{2} \omega_3 - \frac{3}{2} \omega_5$,   \\
			$\mu_{17}(t)= -\frac{5}{2} \omega_{3} + t \omega_5$, & $\mu_{65}(t)=t \omega_1 +(  -\frac{5}{2} - t ) \omega_2 - \frac{1}{2} \omega_4 - \frac{1}{2} \omega_5$,  \\ 
			$\mu_{18}(t)= -\frac{5}{2} \omega_{4} + t \omega_5$, & 	$\mu_{66}(t)=t \omega_1 +(  -1 - t ) \omega_2 - \frac{1}{2} \omega_3 - \frac{1}{2} \omega_4$,  \\  
			$\mu_{19}(t)= -\frac{5}{2} \omega_{1} + t \omega_2 + (-1-t) \omega_3$, & 	$\mu_{67}(t)=t \omega_1 +(  -1 - t ) \omega_2 - \frac{3}{2} \omega_3 - \frac{1}{2} \omega_5$,  \\ 
			$\mu_{20}(t)=-\frac{5}{2} \omega_{1} + t \omega_3 + (-1-t) \omega_4$, & 	$\mu_{68}(t)= t \omega_1 +(  -1 - t ) \omega_2 - \frac{3}{2} \omega_4 + \frac{1}{2} \omega_5$, \\ 
			$\mu_{21}(t)=-\frac{5}{2} \omega_{1} + t \omega_4 + (-1-t) \omega_5$, & 	$\mu_{69}(t)=t \omega_2 +(  -\frac{5}{2} - t ) \omega_3 + \frac{1}{2} \omega_4 - \frac{3}{2} \omega_5$,  \\ 
			$\mu_{22}(t)=\frac{1}{2} \omega_{1} + t \omega_2 + (-\frac{5}{2}-t) \omega_3$, & $\mu_{70}(t)=t \omega_2 +(  -1 - t ) \omega_3 - \frac{3}{2} \omega_4 + \frac{1}{2} \omega_5$, \\ 
			$\mu_{23}(t)= - \frac{1}{2} \omega_{1} + t \omega_3 + (-\frac{5}{2}-t) \omega_4$, & $\mu_{71}(t)=- \frac{3}{2} \omega_1 + \frac{1}{2} \omega_2 + t \omega_3 +(  -\frac{5}{2} - t ) \omega_4$,  \\ 
			$\mu_{24}(t)=- \frac{3}{2} \omega_{1} + t \omega_4 + (-\frac{5}{2}-t) \omega_5$, & $\mu_{72}(t)=- \frac{1}{2} \omega_1 - \frac{1}{2} \omega_2 + t \omega_4 +(  -\frac{5}{2} - t ) \omega_5$,  \\
			$\mu_{25}(t)=-\frac{5}{2} \omega_{2} + t \omega_3 + (-1-t) \omega_4$, & $\mu_{73}(t)=- \frac{3}{2} \omega_1 - \frac{1}{2} \omega_3 + t \omega_4 +(  -\frac{3}{2} - t ) \omega_5$,  \\ 
			$\mu_{26}(t)=-\frac{5}{2} \omega_{2} + t \omega_4 + (-1-t) \omega_5$, & 	$\mu_{74}(t)=- \frac{3}{2} \omega_2 + \frac{1}{2} \omega_3 + t \omega_4 +(  -\frac{5}{2} - t ) \omega_5$,  \\  
			$\mu_{27}(t)=- \frac{1}{2} \omega_{2} + t \omega_3 + (-\frac{3}{2}-t) \omega_4$, & 	$\mu_{75}(t)= \frac{1}{2} \omega_1 - \frac{3}{2} \omega_2 + t \omega_4 +(  -1 - t ) \omega_5$,  
			\\ 
			$\mu_{28}(t)=- \frac{3}{2} \omega_{2} + t \omega_4 + (-\frac{3}{2}-t) \omega_5$, & 	$\mu_{76}(t)=- \frac{1}{2} \omega_1 - \frac{3}{2} \omega_3 + t \omega_4 +(  -1 - t ) \omega_5$, 
			\\ 
			$\mu_{29}(t)=- \frac{5}{2} \omega_{3} + t \omega_4 + (-1-t) \omega_5$, & 	$\mu_{77}(t)=- \frac{1}{2} \omega_2 - \frac{1}{2} \omega_3 + t \omega_4 +(  -1 - t ) \omega_5$,  
			\\ 
			$\mu_{30}(t)= - \frac{3}{2} \omega_{3} + t \omega_4 + (-\frac{1}{2}-t) \omega_5$, & $\mu_{78}(t)= \frac{1}{2} \omega_1 - \frac{3}{2} \omega_2 + t \omega_3 +(  -1 - t ) \omega_4$, 
			\\ 
			$\mu_{31}(t)= t \omega_1 + (-1-t ) \omega_{2} - \frac{5}{2} \omega_5$, & $\mu_{79}(t)=t \omega_1 - \frac{1}{2} \omega_2 - \frac{1}{2} \omega_3  - \frac{1}{2} \omega_4 - \frac{1}{2} \omega_5$,  \\ 
			$\mu_{32}(t)=t \omega_2 + (-1-t ) \omega_{3} - \frac{5}{2} \omega_5$, & $\mu_{80}(t)= -\frac{1}{2} \omega_1-\frac{1}{2} \omega_2 - \frac{1}{2} \omega_3  - \frac{1}{2} \omega_4 +t \omega_5$   
			\\ 
			$\mu_{33}(t)=t \omega_3 + (-1-t ) \omega_{4} - \frac{5}{2} \omega_5$, & $\mu_{81}(t)=t \omega_1 + (-1-t) \omega_2 - \frac{3}{2} \omega_3  + \frac{1}{2} \omega_4 - \frac{3}{2} \omega_5$   
			\\
			$\mu_{34}(t)=t \omega_1 + (-\frac{5}{2}-t ) \omega_{2} - \frac{3}{2} \omega_5$, & $\mu_{82}(t)=t \omega_1 + (-\frac{3}{2}-t) \omega_2 - \frac{1}{2} \omega_3  - \frac{1}{2} \omega_4 - \frac{1}{2} \omega_5$, 
			\\ 
			$\mu_{35}(t)=t \omega_2 + (-\frac{5}{2}-t ) \omega_{3} - \frac{1}{2} \omega_5$, & 	$\mu_{83}(t)= - \frac{3}{2} \omega_1  + \frac{1}{2} \omega_2 - \frac{3}{2} \omega_3 + t \omega_4 + (-1-t) \omega_5$,  
			\\  
			$\mu_{36}(t)=t \omega_3 + (-\frac{5}{2}-t ) \omega_{4} + \frac{1}{2} \omega_5$, & 	$\mu_{84}(t)=- \frac{1}{2} \omega_1  - \frac{1}{2} \omega_2 - \frac{1}{2} \omega_3 + t \omega_4 + (-\frac{3}{2}-t) \omega_5$,  
			\\ 
			$\mu_{37}(t)=t \omega_1 + (-1-t ) \omega_{2} - \frac{5}{2} \omega_4$, & 	$\mu_{85}(t)= -\frac{3}{2} \omega_2 + t \omega_3 + (-1-t) \omega_4 - \frac{1}{2}\omega_5$, 
			\\ 
			$\mu_{38}(t)=t \omega_2 + (-1-t ) \omega_{3} - \frac{5}{2} \omega_4$, & 	$\mu_{86}(t)= -\frac{3}{2} \omega_1 + t \omega_3 + (-1-t) \omega_4 - \frac{3}{2}\omega_5$,  
			\\ 
			$\mu_{39}(t)= t \omega_1 + (-\frac{3}{2}-t ) \omega_{2} - \frac{3}{2} \omega_4$, & $\mu_{87}(t)= -\frac{3}{2} \omega_2 + t \omega_3 + (-\frac{1}{2}-t) \omega_4 - \frac{3}{2}\omega_5$, 
			\\ 
			$\mu_{40}(t)= t \omega_2 + (-\frac{3}{2}-t ) \omega_{3} - \frac{1}{2} \omega_4 $, & $\mu_{88}(t)=-\frac{1}{2} \omega_1 -\frac{1}{2} \omega_2 + t \omega_3 + (-1-t) \omega_4 - \frac{3}{2}\omega_5$,  
			\\ 
			$\mu_{41}(t)= t \omega_1 + (-1-t ) \omega_{2} - \frac{5}{2} \omega_3 $, & $\mu_{89}(t)= -\frac{3}{2} \omega_1 + t \omega_3 + (-\frac{3}{2}-t) \omega_4 - \frac{1}{2}\omega_5$,   
			\\ 
			$\mu_{42}(t)=t \omega_1 + (-\frac{1}{2}-t ) \omega_{2} - \frac{3}{2} \omega_3$, & $\mu_{90}(t)=-\frac{1}{2} \omega_1 -\frac{1}{2} \omega_2 + t \omega_3 + (-\frac{3}{2}-t) \omega_4 - \frac{1}{2}\omega_5$, 
			\\
			$\mu_{43}(t)= t \omega_1 + \frac{1}{2}\omega_2 -\frac{3}{2} \omega_3$, & $\mu_{91}(t) =-\frac{1}{2} \omega_1 + t \omega_2 + (-1-t) \omega_3 - \frac{3}{2}\omega_4$ ,
			\\  
			$\mu_{44}(t)= t \omega_1 - \frac{1}{2}\omega_2 -\frac{3}{2} \omega_4$, & $\mu_{92}(t) = -\frac{3}{2} \omega_1 + t \omega_2 + (-1-t) \omega_3 - \frac{3}{2}\omega_5$,
			\\  
			$\mu_{45}(t)= t \omega_1 - \frac{3}{2}\omega_2 -\frac{3}{2} \omega_5$, & $\mu_{93}(t) = -\frac{3}{2} \omega_1 + t \omega_2 + (-\frac{1}{2}-t) \omega_3 - \frac{3}{2}\omega_4$,
			\\  
			$\mu_{46}(t)= t \omega_1 - \frac{1}{2}\omega_3 -\frac{1}{2} \omega_4$, & $\mu_{94}(t) = -\frac{3}{2} \omega_1 + t \omega_2 + (-1-t) \omega_3 -  \frac{1}{2}\omega_4 - \frac{1}{2}\omega_5$,
			\\  
			$\mu_{47}(t)=  t \omega_1 - \frac{3}{2}\omega_3 -\frac{1}{2} \omega_5$, & $\mu_{95}(t) = -\frac{1}{2} \omega_1 + t \omega_2 + (-\frac{3}{2}-t) \omega_3 - \frac{3}{2}\omega_5$,
			\\  
			$\mu_{48}(t)=  t \omega_1 - \frac{3}{2}\omega_4 +\frac{1}{2} \omega_5$, & $\mu_{96}(t) = -\frac{1}{2} \omega_1 + t \omega_2 + (-\frac{3}{2}-t) \omega_3 -\frac{1}{2} \omega_4 - \frac{1}{2}\omega_5$.
		\end{longtable}
	\end{theorem}
	
	\begin{proof}
		The Proposition \ref{koro-polinomi} and Lemma \ref{lemma-polinomi} imply that the highest weights $\mu\in \h^*$ of irreducible  $\widetilde{L}_{-7/2}(\g)$–modules $L_{-7/2}(\mu)$ in category $\mathcal{O}$ are in one-to-one correspondence with solutions of the system of polynomial equations \begin{equation}\label{eq-sustav}
			p_1(\mu(h))= \cdots = p_9(\mu(h))=0.
		\end{equation}
		Using \textit{Mathematica}, we obtain that solutions of (\ref{eq-sustav}) are exactly weights defined by the set~(\ref{eq-hw-katO}). 
	\end{proof}

	
	

	\begin{corollary}\label{cor-KL}
		The set 
		$$ \{ L_{-7/2}(t \omega_1) \ | \ t \in \Z_{\geq 0} \} \cup  \{ L_{-7/2}(t \omega_5) \ | \ t \in \Z_{\geq 0} \} $$
		provides a complete list of irreducible $\widetilde{L}_{-7/2}(\g)$--modules in category $KL_{-7/2}$.	
	\end{corollary}

	\begin{proposition}
		The singular vector $v$ from Theorem \ref{sing-v-sl6} generates the maximal ideal in $V^{-7/2}(\g)$, i.e. $L_{-7/2}(\g) = V^{-7/2}(\g) / \langle v \rangle$.
	\end{proposition}
	
	\begin{proof} In \cite[Section 5]{APV} it was proved that the maximal ideal in $V^{-5/2}(\mathfrak{sl}_4)$ is generated by one singular vector of conformal weight four. The claim now follows from Theorem \ref{main1} and Corollary \ref{cor-KL}.
	\end{proof}
	
	\begin{remark}
		Since $\widetilde{L}_{-7/2}(\g)$ is simple, Theorem \ref{prop-Zhu-O} gives a classification of irreducible $L_{-7/2}(\g)$--modules in category $\mathcal{O}$. A classification of  $L_{-7/2}(\g)$--modules in category $KL_{-7/2}$ is known from \cite[Theorem 4.1]{ACPV-22}.
	\end{remark}
\begin{remark}
	Conjectural classification of irreducible $L_{-7/2}(\g)$--modules in category $\mathcal{O}$ was given in the Ph.D. Thesis of the third-named author \cite{V22}.
\end{remark}
	
	\section{The case $V^{-9/2}(\mathfrak{sl}_8)$}\label{sl8}
	Throughout this section we denote $\g = \mathfrak{sl}_8$. We prove that in the case  $V^{-9/2}(\g)$ the maximal ideal is also generated by one singular vector of conformal weight four. Since both its explicit formula and polynomials spanning set $\mathcal{P}_0$ are very complicated, we briefly give here only a sketch of the proof.
	\begin{theorem}\label{sing-v-sl8}
		There is a singular vector $v$ in $V^{-9/2}(\g)$ of weight $-\frac{9}{2}\Lambda_0 - 4 \delta + \omega_2 + \omega_6$. Its explicit formula can be found in \textit{Mathematica} file \textit{\href{https://www.dropbox.com/scl/fi/upc8927uq761m8ea3bzpu/sl-8.pdf?rlkey=52r6wzsnl4i8ynowraehyy0wr&dl=0}{ sl(8).nb}}.
	\end{theorem}
	
	\begin{proof}
		Direct verification of relations $e_{i,i+1}(0)v=0$ for $i=1,\ldots,7$ and $f_{1,8}(1)v=0$ using \textit{Mathematica}.
	\end{proof}
	
	Let us denote by $\widetilde{L}_{-9/2}(\g) = V^{-9/2}(\g)/\langle v \rangle$ the associated quotient vertex algebra. Analogous to the case $V^{-7/2}(\mathfrak{sl}_6)$, we determine Zhu's algebra of $\widetilde{L}_{-9/2}(\g)$ and classify irreducible modules for $\widetilde{L}_{-9/2}(\g)$ in the categories $\mathcal{O}$ and $KL_{-9/2}$. In this case, a dimension of the set $\mathcal{P}_0$ defined in Subsection \ref{subsec-Zhu} is $20$ (cf. \cite{AP-08}). Explicit formulas for polynomials spanning $\mathcal{P}_0$ and highest weights of irreducible $\widetilde{L}_{-9/2}(\g)$--modules in the category $\mathcal{O}$ are given in \textit{Mathematica} file  \textit{\href{https://www.dropbox.com/scl/fi/upc8927uq761m8ea3bzpu/sl-8.pdf?rlkey=52r6wzsnl4i8ynowraehyy0wr&dl=0}{ sl(8).nb}}.  From above we have the assertion (1) of the following proposition.

	\begin{proposition} We have:
		\begin{itemize}
			\item[(1)] 	The set $$\{ L_{-9/2}(t \omega_1) \ | \ t \in \Z_{\geq 0} \} \cup \{  L_{-9/2}(t \omega_7) \ | \ t \in \Z_{\geq 0} \} $$
			provides a complete list of irreducible $\widetilde{L}_{-9/2}(\g)$--modules in the category $KL_{-9/2}$.
			\item[(2)] 	The singular vector $v$ from Theorem \ref{sing-v-sl8} generates the maximal ideal in $V^{-9/2}(\g)$, i.e. $L_{-9/2}(\g) = V^{-9/2}(\g) / \langle v \rangle$.
		\end{itemize}
	\end{proposition}
	
	\begin{proof}
		The proof of assertion (2) follows from the first assertion and Theorem \ref{main1}.
		\end{proof}

	\section*{Appendix. Singular vector from Theorem \ref{sing-v-sl6}}
	An explicit formula for the singular vector $v$ from Theorem \ref{sing-v-sl6} is given by:
	\vspace*{-3mm}
	\begin{eqnarray*}
		&&v=e_{1,5}(-3) e_{2,6}(-1){\bf{1} } + \frac{5}{2} e_{1,5}(-2) e_{2,6}(-2){\bf{1} } 
		- e_{1,6}(-3) e_{2,5}(-1){\bf{1} } - \frac{5}{2} e_{1,6}(-2) e_{2,5}(-2){\bf{1} }\\
		&&- e_{2,5}(-3) e_{1,6}(-1) {\bf{1} } + e_{2,6}(-3) e_{1,5}(-1) {\bf{1} } -\frac{7}{5}e_{1,5}(-2) e_{1,6}(-1) f_{1,2}(-1) {\bf{1} }-\frac{5}{3}e_{1,5}(-2)e_{2,3}(-1)e_{3,6}(-1){\bf{1} }\\
		&&- \frac{5}{3} e_{1,5}(-2) e_{2,4}(-1) e_{4,6}(-1) {\bf{1} } - 
		\frac{5}{3} e_{1,5}(-2) e_{2,5}(-1) e_{5,6}(-1) {\bf{1} } - 
		e_{1,5}(-2) e_{2,6}(-1) h_2(-1) {\bf{1} } \\
		&&- \frac{1}{3} e_{1,5}(-2) e_{2,6}(-1) h_3(-1) {\bf{1} } + 
		\frac{1}{3}  e_{1,5}(-2) e_{2,6}(-1) h_4(-1) {\bf{1} } + 
		e_{1,5}(-2) e_{2,6}(-1) h_5(-1) {\bf{1} } \\
		&&- \frac{2}{5}  e_{1,5}(-1) e_{2,6}(-1) h_2(-2) {\bf{1} } + 
		\frac{2}{15} e_{1,5}(-1) e_{2,6}(-1) h_3(-2) {\bf{1} } - 
		\frac{2}{5} e_{1,5}(-1) e_{2,6}(-1) h_4(-2) {\bf{1} }\\
		&& + \frac{7}{5} e_{1,6}(-2) e_{1,5}(-1) f_{1,2}(-1) {\bf{1} } + 
		\frac{5}{3} e_{1,6}(-2) e_{2,3}(-1) e_{3,5}(-1) {\bf{1} } + 
		\frac{5}{3} e_{1,6}(-2) e_{2,4}(-1) e_{4,5}(-1)  {\bf{1} } \\
		&&+ e_{1,6}(-2) e_{2,5}(-1) h_2(-1) {\bf{1} } + \frac{1}{3} e_{1,6}(-2) e_{2,5}(-1) h_3(-1) {\bf{1} } - \frac{1}{3} e_{1,6}(-2) e_{2,5}(-1) h_4(-1) {\bf{1} }\\
		&&+ \frac{2}{3} e_{1,6}(-2) e_{2,5}(-1) h_5(-1) {\bf{1} } + \frac{5}{3} e_{1,6}(-2) e_{2,6}(-1) f_{5,6}(-1) {\bf{1} } + \frac{2}{5} e_{1,6}(-1) e_{2,5}(-1) h_2(-2) {\bf{1} } \\
		&&- \frac{2}{15}e_{1,6}(-1) e_{2,5}(-1) h_3(-2) {\bf{1} } + \frac{2}{5} e_{1,6}(-1) e_{2,5}(-1) h_4(-2) {\bf{1} } - \frac{2}{3} e_{2,3}(-2) e_{1,5}(-1) e_{3,6}(-1) {\bf{1} } \\
		&&+ \frac{2}{3} e_{2,3}(-2) e_{1,6}(-1) e_{3,5}(-1) {\bf{1} } - \frac{2}{3} e_{2,4}(-2) e_{1,5}(-1) e_{4,6}(-1) {\bf{1} } + \frac{2}{3} e_{2,4}(-2) e_{1,6}(-1) e_{4,5}(-1) {\bf{1} } \\
		&&+ e_{2,5}(-2) e_{1,2}(-1) e_{2,6}(-1) {\bf{1} } + e_{2,5}(-2) e_{1,3}(-1) e_{3,6}(-1) {\bf{1} } + e_{2,5}(-2) e_{1,4}(-1) e_{4,6}(-1) {\bf{1} } \\
		&&+ \frac{1}{3} e_{2,5}(-2) e_{1,5}(-1) e_{5,6}(-1) {\bf{1} } + e_{2,5}(-2) e_{1,6}(-1) h_1(-1) {\bf{1} } + e_{2,5}(-2) e_{1,6}(-1) h_2(-1) {\bf{1} } \\
		&&+ \frac{1}{3} e_{2,5}(-2) e_{1,6}(-1) h_3(-1) {\bf{1} } - 
		\frac{1}{3} e_{2,5}(-2) e_{1,6}(-1) h_4(-1) {\bf{1} } - 
		\frac{1}{3} e_{2,5}(-2) e_{1,6}(-1) h_5(-1) {\bf{1} } \\
		&&- e_{2,6}(-2) e_{1,2}(-1) e_{2,5}(-1) {\bf{1} } - e_{2,6}(-2) e_{1,3}(-1) e_{3,5}(-1) {\bf{1} } - e_{2,6}(-2) e_{1,4}(-1) e_{4,5}(-1) {\bf{1} } \\
		&&- e_{2,6}(-2) e_{1,5}(-1) h_1(-1) {\bf{1} } - e_{2,6}(-2) e_{1,5}(-1) h_2(-1) {\bf{1} } - \frac{1}{3} e_{2,6}(-2) e_{1,5}(-1) h_3(-1) {\bf{1} } \\
		&&+ \frac{1}{3} e_{2,6}(-2) e_{1,5}(-1) h_4(-1) {\bf{1} } - 
		\frac{1}{3} e_{2,6}(-2) e_{1,6}(-1) f_{5,6}(-1) {\bf{1}} - 
		\frac{2}{3} e_{3,5}(-2) e_{1,3}(-1) e_{2,6}(-1) {\bf{1}} \\
		&&+ 
		\frac{2}{3} e_{3,5}(-2) e_{1,6}(-1) e_{2,3}(-1) {\bf{1}} + 
		\frac{2}{3} e_{3,6}(-2) e_{1,3}(-1) e_{2,5}(-1) {\bf{1}} - 
		\frac{2}{3} e_{3,6}(-2) e_{1,5}(-1) e_{2,3}(-1) {\bf{1}} \\
		&&+ 
		\frac{2}{3} e_{4,5}(-2) e_{1,4}(-1) e_{2,6}(-1) {\bf{1}} - 
		\frac{2}{3} e_{4,5}(-2) e_{1,6}(-1) e_{2,4}(-1) {\bf{1}} - 
		\frac{2}{3} e_{4,6}(-2) e_{1,4}(-1) e_{2,5}(-1) {\bf{1}} \\
		&&+ 
		\frac{2}{3} e_{4,6}(-2) e_{1,5}(-1) e_{2,4}(-1) {\bf{1}} - 
		\frac{2}{5} e_{1,2}(-1) e_{1,5}(-1) e_{2,6}(-1) f_{1,2}(-1) {\bf{1}} \\
		&&+ 
		\frac{2}{5} e_{1,2}(-1) e_{1,6}(-1) e_{2,5}(-1) f_{1,2}(-1) {\bf{1}} + 
		\frac{2}{3} e_{1,2}(-1) e_{2,3}(-1) e_{2,5}(-1) e_{3,6}(-1) {\bf{1}} \\
		&&- 
		\frac{2}{3} e_{1,2}(-1) e_{2,3}(-1) e_{2,6}(-1) e_{3,5}(-1)  {\bf{1}} + 
		\frac{2}{3} e_{1,2}(-1) e_{2,4}(-1) e_{2,5}(-1) e_{4,6}(-1) {\bf{1}} \\
		&&- 
		\frac{2}{3} e_{1,2}(-1) e_{2,4}(-1) e_{2,6}(-1) e_{4,5}(-1) {\bf{1}} + 
		\frac{2}{3} e_{1,2}(-1) e_{2,5}(-1)^2 e_{5,6}(-1) {\bf{1}} \\
		&&-
		\frac{2}{3} e_{1,2}(-1) e_{2,5}(-1) e_{2,6}(-1) h_5(-1) {\bf{1}} - \frac{2}{3} e_{1,2}(-1) e_{2,6}(-1)^2 f_{5,6}(-1) {\bf{1}} \\
		&&+ 
		\frac{4}{15} e_{1,3}(-1) e_{1,5}(-1) e_{2,6}(-1) f_{1,3}(-1) {\bf{1}} - \frac{2}{3} e_{1,3}(-1)  e_{1,5}(-1) e_{2,6}(-1) f_{1,2}(-1){\bf{1}} \\
		&&- 
		\frac{4}{15} e_{1,3}(-1) e_{1,6}(-1) e_{2,5}(-1) f_{1,3}(-1) {\bf{1}} + \frac{2}{3}  e_{1,3}(-1) e_{1,6}(-1) e_{3,5}(-1) f_{1,2}(-1) {\bf{1}} \\
		&&+ 
		2 e_{1,3}(-1) e_{2,4}(-1) e_{3,5}(-1) e_{4,6}(-1) {\bf{1}} - 
		2 e_{1,3}(-1) e_{2,4}(-1) e_{3,6}(-1) e_{4,5}(-1) {\bf{1}} \\
		&&- 
		\frac{4}{3} e_{1,3}(-1) e_{2,5}(-1) e_{3,4}(-1) e_{4,6}(-1) {\bf{1}} + \frac{2}{3} e_{1,3}(-1) e_{2,5}(-1) e_{3,5}(-1) e_{5,6}(-1){\bf{1}} \\
		&&- 
		\frac{2}{3} e_{1,3}(-1) e_{2,5}(-1) e_{3,6}(-1) h_2(-1) {\bf{1}} - 
		\frac{2}{3} e_{1,3}(-1) e_{2,5}(-1) e_{3,6}(-1) h_3(-1) {\bf{1}} \\
		&&+ 
		\frac{2}{3} e_{1,3}(-1) e_{2,5}(-1) e_{3,6}(-1) h_4(-1) {\bf{1}} + 
		\frac{4}{3} e_{1,3}(-1) e_{2,6}(-1) e_{3,4}(-1) e_{4,5}(-1) {\bf{1}} \\
		&&+ 
		\frac{2}{3} e_{1,3}(-1) e_{2,6}(-1) e_{3,5}(-1) h_2(-1) {\bf{1}} + \frac{2}{3} e_{1,3}(-1) e_{2,6}(-1) e_{3,5}(-1) h_3(-1) {\bf{1}} \\
		&&- 
		\frac{2}{3} e_{1,3}(-1) e_{2,6}(-1) e_{3,5}(-1) h_4(-1) {\bf{1}} - \frac{2}{3} e_{1,3}(-1) e_{2,6}(-1) e_{3,5}(-1) h_5(-1) {\bf{1}} \\
		&&- 
		\frac{2}{3} e_{1,3}(-1) e_{2,6}(-1) e_{3,6}(-1) f_{5,6}(-1) {\bf{1}} + \frac{4}{15} e_{1,4}(-1) e_{1,5}(-1) e_{2,6}(-1) f_{1,4}(-1) {\bf{1}} \\
		&&- 
		\frac{2}{3} e_{1,4}(-1) e_{1,5}(-1) e_{4,6}(-1) f_{1,2}(-1) {\bf{1}} - \frac{4}{15} e_{1,4}(-1) e_{1,6}(-1) e_{2,5}(-1) f_{1,4}(-1) {\bf{1}}  \\
		&&+
		\frac{2}{3} e_{1,4}(-1) e_{1,6}(-1) e_{4,5}(-1) f_{1,2}(-1) {\bf{1}} - 
		2 e_{1,4}(-1) e_{2,3}(-1) e_{3,5}(-1) e_{4,6}(-1) {\bf{1}} \\
		&&+ 
		2 e_{1,4}(-1) e_{2,3}(-1) e_{3,6}(-1) e_{4,5}(-1) {\bf{1}} - 
		\frac{4}{3} e_{1,4}(-1) e_{2,5}(-1) e_{3,6}(-1) f_{3,4}(-1) {\bf{1}} \\
		&&+ 
		\frac{2}{3} e_{1,4}(-1) e_{2,5}(-1) e_{4,5}(-1) e_{5,6}(-1) {\bf{1}} - 
		\frac{2}{3} e_{1,4}(-1) e_{2,5}(-1) e_{4,6}(-1) h_2(-1) {\bf{1}} \\
		&&+
		\frac{2}{3} e_{1,4}(-1) e_{2,5}(-1) e_{4,6}(-1) h_3(-1) {\bf{1}} + \frac{2}{3} e_{1,4}(-1) e_{2,5}(-1) e_{4,6}(-1) h_4(-1) {\bf{1}} \\
		&&+ 
		\frac{4}{3} e_{1,4}(-1) e_{2,6}(-1) e_{3,5}(-1) f_{3,4}(-1) {\bf{1}}+ 
		\frac{2}{3} e_{1,4}(-1) e_{2,6}(-1) e_{4,5}(-1) h_2(-1) {\bf{1}} \\
		&&- 
		\frac{2}{3} e_{1,4}(-1) e_{2,6}(-1) e_{4,5}(-1) h_3(-1) {\bf{1}} - \frac{2}{3} e_{1,4}(-1) e_{2,6}(-1) e_{4,5}(-1) h_4(-1) {\bf{1}} \\
		&&- 
		\frac{2}{3} e_{1,4}(-1) e_{2,6}(-1) e_{4,5}(-1) h_5(-1) {\bf{1}} - \frac{2}{3} e_{1,4}(-1) e_{2,6}(-1) e_{4,6}(-1) f_{5,6}(-1) {\bf{1}} \\
		&&+ 
		\frac{4}{15} e_{1,5}(-1)^2  e_{2,6}(-1) f_{1,5}(-1) {\bf{1}} - 
		\frac{2}{3} e_{1,5}(-1)^2  e_{5,6}(-1) f_{1,2}(-1) {\bf{1}} \\
		&&- 
		\frac{4}{15} e_{1,5}(-1) e_{1,6}(-1) e_{2,5}(-1) f_{1,5}(-1) {\bf{1}} + 
		\frac{4}{15} e_{1,5}(-1) e_{1,6}(-1) e_{2,6}(-1) f_{1,6}(-1) {\bf{1}} \\
		&&+ \frac{2}{3} e_{1,5}(-1) e_{1,6}(-1) f_{1,2}(-1) h_5(-1) {\bf{1}} + \frac{4}{15} e_{1,5}(-1) e_{2,3}(-1) e_{2,6}(-1) f_{2,3}(-1) {\bf{1}} \\
		&&+ 
		\frac{4}{3} e_{1,5}(-1) e_{2,3}(-1) e_{3,4}(-1) e_{4,6}(-1) {\bf{1}} - \frac{2}{3} e_{1,5}(-1) e_{2,3}(-1) e_{3,5}(-1) e_{5,6}(-1) {\bf{1}} \\
		&&+ 
		\frac{2}{3} e_{1,5}(-1) e_{2,3}(-1) e_{3,6}(-1) h_1(-1) {\bf{1}} + 
		\frac{2}{3} e_{1,5}(-1) e_{2,3}(-1) e_{3,6}(-1) h_2(-1) {\bf{1}}\\
		&& + 
		\frac{2}{3} e_{1,5}(-1) e_{2,3}(-1) e_{3,6}(-1) h_3(-1) {\bf{1}} - 
		\frac{2}{3} e_{1,5}(-1) e_{2,3}(-1) e_{3,6}(-1) h_4(-1) {\bf{1}} \\
		&&+ 
		\frac{4}{15} e_{1,5}(-1) e_{2,4}(-1) e_{2,6}(-1) f_{2,4}(-1) {\bf{1}} 
		+ \frac{4}{3} e_{1,5}(-1) e_{2,4}(-1) e_{3,6}(-1) f_{3,4}(-1) {\bf{1}} \\
		&&
		- \frac{2}{3} e_{1,5}(-1) e_{2,4}(-1) e_{4,5}(-1) e_{5,6}(-1) {\bf{1}} 
		+ \frac{2}{3} e_{1,5}(-1) e_{2,4}(-1) e_{4,6}(-1) h_1(-1) {\bf{1}} \\
		&&
		+ \frac{2}{3} e_{1,5}(-1) e_{2,4}(-1) e_{4,6}(-1) h_2(-1) {\bf{1}}
		- \frac{2}{3} e_{1,5}(-1) e_{2,4}(-1) e_{4,6}(-1) h_3(-1) {\bf{1}} \\
		&&
		- \frac{2}{3} e_{1,5}(-1) e_{2,4}(-1) e_{4,6}(-1) h_4(-1) {\bf{1}} 
		+ \frac{4}{15} e_{1,5}(-1) e_{2,5}(-1) e_{2,6}(-1) f_{2,5}(-1) {\bf{1}} \\
		&&
		+ \frac{2}{3} e_{1,5}(-1) e_{2,5}(-1) e_{5,6}(-1) h_1(-1) {\bf{1}} 
		+ \frac{4}{15} e_{1,5}(-1) e_{2,6}(-1) e_{2,6}(-1) f_{2,6}(-1) {\bf{1}} \\
		&&
		- \frac{16}{15} e_{1,5}(-1) e_{2,6}(-1) e_{3,4}(-1) f_{3,4}(-1) {\bf{1}} 
		+ \frac{4}{15} e_{1,5}(-1) e_{2,6}(-1) e_{3,5}(-1) f_{3,5}(-1) {\bf{1}} \\
		&&
		+ \frac{4}{15} e_{1,5}(-1) e_{2,6}(-1) e_{3,6}(-1) f_{3,6}(-1) {\bf{1}} 
		+ \frac{4}{15} e_{1,5}(-1) e_{2,6}(-1) e_{4,5}(-1) f_{4,5}(-1) {\bf{1}} \\
		&&
		+ \frac{4}{15} e_{1,5}(-1) e_{2,6}(-1) e_{4,6}(-1) f_{4,6}(-1) {\bf{1}} 
		- \frac{2}{5}  e_{1,5}(-1) e_{2,6}(-1) e_{5,6}(-1) f_{5,6}(-1) {\bf{1}} \\
		&&
		+ \frac{2}{5} e_{1,5}(-1) e_{2,6}(-1) h_1(-1) h_2(-1) {\bf{1}} + 
		\frac{2}{15} e_{1,5}(-1) e_{2,6}(-1) h_1(-1) h_3(-1) {\bf{1}} \\
		&&- \frac{2}{15} e_{1,5}(-1) e_{2,6}(-1) h_1(-1) h_4(-1) {\bf{1}} - 
		\frac{2}{5} e_{1,5}(-1) e_{2,6}(-1) h_1(-1) h_5(-1) {\bf{1}} \\
		&&
		+ \frac{2}{5} e_{1,5}(-1) e_{2,6}(-1) h_2(-1) h_2(-1) {\bf{1}} 
		+ \frac{4}{15} e_{1,5}(-1) e_{2,6}(-1) h_2(-1) h_3(-1) {\bf{1}} \\
		&&
		- \frac{4}{15} e_{1,5}(-1) e_{2,6}(-1) h_2(-1) h_4(-1) {\bf{1}} 
		- \frac{2}{15} e_{1,5}(-1) e_{2,6}(-1) h_2(-1) h_5(-1) {\bf{1}} \\
		&&
		- \frac{2}{15} e_{1,5}(-1) e_{2,6}(-1) h_3(-1) h_3(-1) {\bf{1}} 
		+ \frac{4}{15} e_{1,5}(-1) e_{2,6}(-1) h_3(-1) h_4(-1) {\bf{1}} \\
		&&
		+ \frac{2}{15} e_{1,5}(-1) e_{2,6}(-1) h_3(-1) h_5(-1) {\bf{1}} 
		+ \frac{2}{5} e_{1,5}(-1) e_{2,6}(-1) h_4(-1) h_4(-1) {\bf{1}} \\
		&&
		+ \frac{2}{5} e_{1,5}(-1) e_{2,6}(-1) h_4(-1) h_5(-1) {\bf{1}} 
		- \frac{4}{15} e_{1,6}(-1)^2 e_{2,5}(-1) f_{1,6}(-1) {\bf{1}} \\
		&&
		+ \frac{2}{3} e_{1,6}(-1)^2 f_{1,2}(-1) f_{5,6}(-1) {\bf{1}} 
		- \frac{4}{15} e_{1,6}(-1) e_{2,3}(-1) e_{2,5}(-1) f_{2,3}(-1) {\bf{1}} \\
		&&
		-\frac{4}{3} e_{1,6}(-1) e_{2,3}(-1)  e_{3,4}(-1) e_{4,5}(-1) {\bf{1}} - 
		\frac{2}{3} e_{1,6}(-1) e_{2,3}(-1) e_{3,5}(-1) h_1(-1) {\bf{1}} \\
		&&
		- \frac{2}{3} e_{1,6}(-1) e_{2,3}(-1) e_{3,5}(-1) h_2(-1)  {\bf{1}} 
		- \frac{2}{3} e_{1,6}(-1) e_{2,3}(-1) e_{3,5}(-1) h_3(-1)  {\bf{1}} \\
		&&
		+ \frac{2}{3} e_{1,6}(-1) e_{2,3}(-1) e_{3,5}(-1) h_4(-1)  {\bf{1}} 
		+ \frac{2}{3} e_{1,6}(-1) e_{2,3}(-1) e_{3,5}(-1) h_5(-1)  {\bf{1}} \\
		&&
		+ \frac{2}{3} e_{1,6}(-1) e_{2,3}(-1) e_{3,6}(-1) f_{5,6}(-1) {\bf{1}} 
		- \frac{4}{15} e_{1,6}(-1) e_{2,4}(-1) e_{2,5}(-1) f_{2,4}(-1) {\bf{1}} \\
		&&
		- 
		\frac{4}{3} e_{1,6}(-1) e_{2,4}(-1) e_{3,5}(-1) f_{3,4}(-1) {\bf{1}} 
		- \frac{2}{3} e_{1,6}(-1) e_{2,4}(-1) e_{4,5}(-1) h_1(-1) {\bf{1}} \\
		&&
		- \frac{2}{3} e_{1,6}(-1) e_{2,4}(-1) e_{4,5}(-1) h_2(-1) {\bf{1}} 
		+ \frac{2}{3} e_{1,6}(-1) e_{2,4}(-1) e_{4,5}(-1) h_3(-1) {\bf{1}} \\
		&&
		+ \frac{2}{3} e_{1,6}(-1) e_{2,4}(-1) e_{4,5}(-1) h_4(-1) {\bf{1}} 
		+ \frac{2}{3} e_{1,6}(-1) e_{2,4}(-1) e_{4,5}(-1) h_5(-1) {\bf{1}} \\
		&&
		+ \frac{2}{3} e_{1,6}(-1) e_{2,4}(-1) e_{4,6}(-1) f_{5,6}(-1) {\bf{1}} 
		- \frac{4}{15} e_{1,6}(-1) e_{2,5}(-1)^2  f_{2,5}(-1) {\bf{1}} \\
		&&
		- \frac{4}{15} e_{1,6}(-1)  e_{2,5}(-1) e_{2,6}(-1) f_{2,6}(-1){\bf{1}}
		+ \frac{16}{15} e_{1,6}(-1) e_{2,5}(-1) e_{3,4}(-1) f_{3,4}(-1) {\bf{1}}\\
		&& 
		- \frac{4}{15} e_{1,6}(-1) e_{2,5}(-1) e_{3,5}(-1) f_{3,5}(-1){\bf{1}} 
		- \frac{4}{15} e_{1,6}(-1) e_{2,5}(-1) e_{3,6}(-1) f_{3,6}(-1) {\bf{1}} \\
		&&
		- \frac{4}{15} e_{1,6}(-1) e_{2,5}(-1) e_{4,5}(-1) f_{4,5}(-1) {\bf{1}} 
		- \frac{4}{15} e_{1,6}(-1) e_{2,5}(-1) e_{4,6}(-1) f_{4,6}(-1) {\bf{1}} \\
		&&
		+ \frac{2}{5} e_{1,6}(-1) e_{2,5}(-1)e_{5,6}(-1) f_{6,6}(-1) {\bf{1}} 
		- \frac{2}{5}  e_{1,6}(-1) e_{2,5}(-1) h_1(-1) h_2(-1) {\bf{1}} \\
		&&
		- \frac{2}{15} e_{1,6}(-1) e_{2,5}(-1) h_1(-1) h_3(-1) {\bf{1}}
		+ \frac{2}{15} e_{1,6}(-1) e_{2,5}(-1) h_1(-1) h_4(-1) {\bf{1}} \\
		&&
		- \frac{4}{15} e_{1,6}(-1) e_{2,5}(-1) h_1(-1) h_5(-1) {\bf{1}} 
		- \frac{2}{5}  e_{1,6}(-1) e_{2,5}(-1) h_2(-1) h_2(-1) {\bf{1}}\\
		&&
		- \frac{4}{15} e_{1,6}(-1) e_{2,5}(-1) h_2(-1) h_3(-1) {\bf{1}} 
		+ \frac{4}{15} e_{1,6}(-1) e_{2,5}(-1) h_2(-1) h_4(-1) {\bf{1}} \\
		&&
		+ \frac{2}{15} e_{1,6}(-1) e_{2,5}(-1) h_2(-1) h_5(-1) {\bf{1}} 
		+ \frac{2}{15} e_{1,6}(-1) e_{2,5}(-1) h_3(-1) h_3(-1) {\bf{1}} \\
		&&
		- \frac{4}{15} e_{1,6}(-1) e_{2,5}(-1) h_3(-1) h_4(-1) {\bf{1}} 
		- \frac{2}{15} e_{1,6}(-1) e_{2,5}(-1) h_3(-1) h_5(-1) {\bf{1}} \\
		&&
		- \frac{2}{5} e_{1,6}(-1) e_{2,5}(-1) h_4(-1) h_4(-1) {\bf{1}} 
		- \frac{2}{5} e_{1,6}(-1) e_{2,5}(-1) h_4(-1) h_5(-1) {\bf{1}} \\
		&&
		- \frac{2}{3} e_{1,6}(-1) e_{2,6}(-1) f_{5,6}(-1) h_1(-1) {\bf{1}}.
	\end{eqnarray*}

\end{document}